 \documentclass{interact}
 
 \usepackage{amsmath}
 \usepackage{enumerate}
 \usepackage{multirow}
 \usepackage{graphicx}
 \usepackage{amssymb}
 \usepackage{amsthm}
 \usepackage{epstopdf,geometry}
 
 \usepackage{float}
 \usepackage{amsfonts}
 \usepackage{multirow}
 \usepackage{graphicx}
 \newtheorem{example}{Example}[section]
 \usepackage{caption}
 \usepackage{rotfloat} 
 \usepackage{hyperref}
 \usepackage{appendix}
 \usepackage{enumerate}
 \usepackage{mathtools}
 \usepackage{mathrsfs}
 \usepackage{booktabs}
 \newtheorem{assumption}{Assumption}
 \usepackage{algorithm} 
 \usepackage{algorithmic} 

 \usepackage{epstopdf}
 \usepackage[caption=false]{subfig}

 \usepackage[numbers,sort&compress]{natbib}
 \bibpunct[, ]{[}{]}{,}{n}{,}{,}
 
 \theoremstyle{plain}
 \newtheorem{theorem}{Theorem}[section]
 \newtheorem{lemma}[theorem]{Lemma}
 \newtheorem{corollary}[theorem]{Corollary}

 \theoremstyle{definition}
 \newtheorem{definition}[theorem]{Definition}
 \theoremstyle{remark}
 \newtheorem{remark}{Remark}

\begin{document}
 	
 	\articletype{ARTICLE TEMPLATE}
 	
 	\title{Global convergence of block proximal iteratively reweighted algorithm  with
 	extrapolation    }
 	
 	\author{
 		\name{Jie Zhang\textsuperscript{a}, Xinmin Yang\textsuperscript{b}\thanks{CONTACT  X.M. Yang. Author. Email: xmyang@cqnu.edu.cn}}
 		\affil{\textsuperscript{a} College of Mathematics, Sichuan University, Chengdu 610065, China\\
 			\textsuperscript{b} School of Mathematics Science, Chongqing Normal University, Chongqing 401331, China}
 	}
 	
 	\maketitle
 	
\begin{abstract}
   In this paper, 		
 we propose a proximal iteratively reweighted  algorithm with 
 extrapolation based on block coordinate update
 aimed at solving a class of optimization problems which is the sum of a smooth possibly nonconvex loss function and a general nonconvex regularizer with a special structure. The proposed algorithm can be used to solve the $\ell_p(0<p<1)$ regularization problem by employing a updating strategy of the smoothing parameter.  
 It is proved that there exists the nonzero extrapolation parameter such that the objective function is nonincreasing. Moreover, the global
  convergence 
  and local convergence rate are obtained by using the Kurdyka-{\L}ojasiewicz (KL) property on the objective function. 
 Numerical experiments are given to indicate the efficiency of the proposed algorithm.
\end{abstract}

Block coordinate update; Extrapolation; Convergence analysis

  \section{Introduction}
 We mainly concentrate on the following  optimization problem in this paper 
 \begin{align}\label{1.1}
 	\mathop{\min}_{ {x}\in\mathbb{R}^n} \quad F({ {x_{s_1}}},{ {x_{s_2}}},\dots,{ {x_{s_m}}}):=f({ {x_{s_1}}},{ {x_{s_2}}},\dots,{{x_{s_m}}})+\lambda \sum_{i=1}^{m}\sum_{j\in{{{s_i}}}} h(g({{x_{j}}})),\, \lambda>0,
 \end{align}
 where $f,g$ and $h$ are of the below properties.
 \begin{assumption}\label{assumption0}
 	\begin{enumerate}[\rm(i)]
 		\item $f:\mathbb{R}^n \rightarrow\mathbb{R}$ is {a continuously} differentiable function and has a block Lipschitz continuous gradient constant, i.e., for some disjoint sets $\{{ {x_{s_1}}},{ {x_{s_2}}},\dots,{ {x_{s_m}}}\}$  satisfying $\cup{\{ {s_1},{ {{s_1}}},\dots, {s_m}}\}= \{1, 2,\dots,n\}$, it has that 
 		\begin{align}\label{1.02}
 			\Vert \nabla_{{s_j}} f({{x_{s_1}}},\dots,{{x_{s_j}}},\dots, {{x_{s_m}}}) -
 			\nabla_{{s_j}} f({{x_{s_1}}},\dots, {{\hat x_{s_j}}},\dots,
 			{{x_{s_m}}})
 			\Vert \leq L_j\Vert  {x_{s_j}} -  {\hat x_{s_j}}\Vert,
 		\end{align}	
 		and there exist constants $0<\ell\leq L<\infty$ satisfying
 		$\ell\leq L_j\leq L$ for any $j\in \{{{{s_1}}},{ {{s_2}}},\dots,{{{s_m}}}\}; $ 
 		\item  $g:\mathbb{R} \to\mathbb{R}$ is a nonnegative   closed and convex function\,$;$
 		\item $h: \rm{Im(g)}\to \mathbb{R}_+$ is a concave function and satisfies
 		\,$h'(t)>0$,\,
 		$\forall\, t\in {\rm{Im}} (g) $.	
 		Moreover, it is differentiable with a Lipschitz continuous gradient $L_h$, i.e., for any $s, t\in {\rm{Im}} (g)$, 
 		\[ \vert h'(s)-  h'(t)\vert \leq L_h \vert s-t\vert;   \]
 		\item $F$ is coercive, that is when $\Vert x\Vert\to +\infty$, $F\to +\infty$.
 	\end{enumerate}
 \end{assumption}
 
 
 The model (\ref{1.1}) has wide applications which arise in compressive sensing \cite{Chartrand}, variable selections \cite{Chen X}  and signal processing \cite {Daubechies} and so on. The former item of the model   (\ref{1.1}) is generally a loss function which is smooth and may be nonconvex, such as $f(x) =\frac{1}{n}\sum_{i=1}^{n} {\rm{log}}(1+{\rm{exp}}(-b_ia^T_ix))$, 
 where $A=[a^T_1;\dots;a^T_n]\in\mathbb{R}^{n\times d}, b\in\mathbb{R}^n$.
 The later item of the model (\ref{1.1}) can be can be viewed as a regularization or penalty component, such as
 the log penalty function \cite{Candes}, the approximate $\ell_p$  norm \cite{Chen X} and so on.  
 When $s_j = \{j\},$ the problem (\ref{1.1}) reduces to the following  univariate problem,
 \begin{align}\label{1.2}
 	\mathop{\min}_{x\in\mathbb{R}^n} \quad F(x):=f(x)+ \lambda\sum_{i=1}^{n} h(g(x_i),\,\lambda>0.
 \end{align}
 Lu et al.\cite{LuC} have developed  the  proximal iteratively reweighted algorithm (PIRE) for solving the problem (\ref{1.2}),  
 \begin{align}\label{1.3}
 	x^{k+1}=\arg\min_{x}\sum^{n}_{i=1} w^k_ig(x_i)+\frac{L}{2}\Vert x-(x^k-\frac{1}{L}\nabla f(x^k)) \Vert^2,
 \end{align}
 where $w^k_i:=\lambda h'(g(x^k_i))$ and $L>0$.
 It needs to be mentioned that the subproblem (\ref{1.3}) has a  unique solution 
 since the subproblem (\ref{1.3}) is strongly convex. Additionally,   
 the solution has an explicit form  
   when the proximal operator of $g$  is easily calculated.  Only subsequence convergence of the PIRE algorithm is obtained in \cite{LuC}.  
 To deeply
 discuss the convergence of the PIRE algorithm, 
 Sun et al. \cite{Sun T}
 strengthened the 
 condition on the objective, i.e.,
 it is required that $h$ is continuously differentiable and has a Lipschitz continuous gradient on the 
 image set of function $g$ and proved 
 the  global convergence of the
algorithmic sequence  under the  Kurdyka-{\L}ojasiewicz(KL) property on the objective function $F$. 
Their algorithmic framework unifies the 
PIRE, the Proximal Iteratively  Reweighted algorithm with parallel splitting (PIRE-PS)
and the Proximal Iteratively
Reweighted algorithm with alternative updating (PIRE-AU). 


 It is well known that  
 extrapolation technique proposed by
 Nesterov  \cite{Yurii Nesterov}
 has been an effective way to improve the first order algorithms.
Recently,  Yu and Pong  \cite{Yu P}  proposed the  iteratively reweighted  $\ell_1$  algorithm with three different extrapolation ways    
 based on Nesterov's method  and its extentions \cite{Beck, Auslender, Tseng,LanG,Drusvyatskiy,Ghadimi}  
 to solve the following problem,             
 \begin{align}\label{1.5}
 	\min_{x\in\mathbb{R}^n}  F(x):=f(x)+\delta_C(x)+\sum_{i=1}^{n}\phi(\vert y_i\vert),
 \end{align}
 where $f:\mathbb{R}^n\to\mathbb{R}$ is a smooth convex function with a Lipschitz continuous gradient, $C$ is a nonempty closed convex set,   
  $\phi:\mathbb{R}_{+}\to\mathbb{R}_{+}$ is a continuous concave function  and continuously differentiable
 on $(0,+\infty)$ and satisfies $\phi(0)=0$.  
 Moreover,  it is assumed that  $  \lim_{t\to 0}\phi'(t)$ exists. It is observed that the 
 model (\ref{1.1}) and(\ref{1.5}) are not the same. 
 The global convergence is obtained under  KL property on  certain potential functions. 
 
  Benefitted from the advantages of the extrapolation technique, Xu and Yin \cite{XuY1} proposed a block proximal gradient method with   extrapolation  based on the fixed cyclic order of each block to solve the optimization problem that the feasible set and objective function are convex in each block of variables.
  The global convergence of the algorithmic sequence is obtained under the KL condition.    
 Further, Xu and Yin \cite{XuY}  considered 
 the  problem which is the sum of two nonconvex functions 
 and always require the proximal operator of the nonsmooth component is easily
 calculated. 
 They developed a proximal gradient method with the extrapolation term based on block update
 and each block is updated at least once either deterministically cyclic or randomly shuffled for each cycle.
 They proved that there exists a nonzero extrapolation parameter such that the objective function is nonincreasing under certain conditions and get the same convergence  under  
 the KL property.
 
 Witnessing the success of accelerated  methods mentioned above, in this paper, we propose a  
 iteratively reweighted   
 algorithm  with extrapolation 
 based on the block coordinate update to solve problem (\ref{1.1}), in which 
 the proximal operator of the regularization may not have a closed form solution.
  Each block is updated at least once within fixed iterations, either through deterministic cycling or through random shuffles.
  Especially, the proposed algorithm  could be used to solve 
  the $\ell_p (0<p<1)$ regularization problem by using a existing 
  update strategy of smoothing factor.  
    We get that the existence of nonzero
    extrapolation parameters makes the objective function nonincreasing. 
    Further, by imposing the assumption of KL  property on the objective function, we prove that the global convergence and  local convergence rate of the sequence generated by the
    BPIREe algorithm.  
    
    The composition of the paper includes the following parts.
    Some preparatory work is given in Section 2.   
    and in Section 3, we propose a BPIREe algorithm and analyze its convergence behaviour. Further, in Section 4  we apply the proposed algorithm to solve the $\ell_p$ reguralization problem and discuss its convergence. 
    Section 4 illustrates the performance of the proposed algorithm 
    through  numerical experiments. At last, we give the conclusions in Section 5.
    
     \textbf{Notations.} The sets $\mathbb{R}^n$, $\mathbb{R}^n_{+}$ and $\mathbb{R}^n_{++}$ represent $n$-dimensional Euclidean space, positive orthant in $\mathbb{R}^n$ and the interior of $\mathbb{R}^n$ respectively. For $x\in\mathbb{R}^n$,    
    $\Vert x\Vert:=\Vert x\Vert_2
    =(\sum_{i=1}^{n} x^2_i)^\frac{1}{2}$ denotes the
    Euclidean norm.  $\Vert x\Vert_p=(\sum_{i=1}^{n}\vert x_i \vert^p)^{\frac{1}{p}} (0<p<1)$ represents the $\ell_p$ norm.  The 
    Hadamard product for any two vectors $a,b\in\mathbb{R}^n$ is defined as 
    $(a\odot b)_i=a_ib_i$. Denote $\mathbb{N}=\{0,1,\dots,n\}$. Let 
    $\{-1,0, 1\}^n $ be the set of $n$-dimensional vectors with components drawn from  $\{-1,0, 1\}$.  $\rm{Im}(g)$ represents the image set of function $g$. 
    The set $\mathbb{R}^{m\times n}$ is denoted as the space of $m\times n$ matrices. For any $A=[a_{i,j}]_{m\times n}\in\mathbb{R}^{m\times n}$, the Frobenius norm of $X$ is defined ad 
    $\Vert A\Vert_F=\sqrt{\sum_{i,j}a^2_{i,j}}$.
    Denote $A^T$ be the transpose of the matrix $A$.

\section{Preliminaries}
  In this section, we mainly recall some definitions and preliminary results which will be helpful for the analysis of this paper.
 
 The domain of an extended real-valued function  $\phi:\mathbb{R}^n\to (-\infty, +\infty]$ is defined as $ {\rm{dom}}\, \phi: =\{x\in\mathbb{R}^n, \phi(x) <+\infty \}$. If the domain is nonempty, it is said to be proper. The function $\phi$ is said to be closed if it is lower semicontinuous.

  \begin{definition}\label{def2.1}
  	(subdifferential  \cite{Mordukhovich, Rockafellar}) Let
 	$ \phi (x): \mathbb{R}^n\to({-\infty},{+\infty}]$ be a proper and closed function.
 	The limiting-subdifferential (or subdifferential) is the set of all vectors $u\in \mathbb{R}^n$ satisfying
 	\begin{align*}
 		\partial \phi (x):=\left\{u \in \mathbb{R}^n | \exists x^k\to x,\, \phi (x^k)\to \phi (x)\, {\rm{and}} \, u^k \in \hat\partial \phi (x^k),\, u^k\to u \,{\rm{as}}\, k\to +\infty \right\},
 	\end{align*}
 	where  $\hat\partial\phi(x)$ is the Fr\'{e}chet subdifferential of $\phi$ at $x$, if $x\in {\rm{dom} }\phi$, defined by
 	\[
 	\lim_{y\ne x}\inf_{y\to x}\frac{\phi(y)-\phi(x)-\langle u, y-x \rangle}{\Vert x-y \Vert}\ge 0,
 	\]
 	otherwise, we set $\hat\partial \phi(x) =\emptyset$.
 \end{definition}
 If $\phi$ is continuously differentiable, then $\partial\phi= \{\nabla\phi \}$.
  When  $\phi$ is convex,
  the subdifferential reduces to the
 classical subdifferential in the sense of convex analysis
 \[  \partial\phi (x):=\{u: \phi(y)-\phi(x)\ge \langle u, y-x\rangle   \}.  \]

The Kurdyka-{\L}ojasiewicz property \cite{Bolte, Attouch} is useful for the convergence analysis of first order methods, which apply to a wide range of problems, including nonsmooth semialgebraic minimization problems   \cite{Attouch, Hedy Attouch}. 
  
   
  \begin{definition}\label{def1}
  	(Kurdyka-{\L}ojasiewicz property) Let $J(x):\mathbb{R}^n\to({-\infty},{+\infty}]$ be a proper closed function, the function $J(x)$ is said to have the Kurdyka- {\L}ojasiewicz (KL) property at $\bar x\in {\rm{dom}}\partial J:=\left\{x\in \mathbb{R}^n:\partial J(x)\ne \emptyset \right\}$ if there exists a neighborhood $U$ of $\bar x$, $\xi\in(0,+\infty]$, and a function $\phi:[0,\xi)\to\mathbb{R}_+$ satisfying
  	
  	\begin{enumerate}[\rm(i)]
  		\item	$\phi(0)=0,\phi\in C^1(0,\xi)$;	
  		\item for all $s\in(0,\xi),{\phi}'(s)>0$;	
  		\item  for any
  		$x\in U\cap[J(\bar x)<J(x)<J({\bar x})+\xi]$,
  	\end{enumerate}
  	 it holds that	
  	\begin{align} \label{A1}
  		\phi{'}(J(x)-J(\bar x))\cdot {\rm{dist}}(0,\partial J(x))\ge1.
  	\end{align}
  	 
  	If $J$ satisfies the KL property at each point of ${\rm{dom}} J$, then $J$ is called a KL function. If $J$ satisfies the KL property at $\bar x$ and $\phi$ in (\ref{A1}) is chosen as $\phi(s) = cs^{1-\theta}$ for some $\theta\in[0,1)$, $c>0$, then we can say that $J$ satisfies the KL property at $\bar x$ with exponent $\alpha$.
  \end{definition}
  If $f$ is smooth, then the inequality (\ref{A1}) reduces to the following  given by Attouch et al. \cite{Hedy Attouch}
  \[  \Vert \nabla (\phi \circ J ) \Vert\ge 1.\]
%
  \begin{lemma}{\rm({\cite{Rockafellar}})}\label{lemma 2.4}
 Under the Assumption \ref{assumption0}(ii), there exists a $L_g>0$ such that
 \[  \vert g (x)- g(y) \vert \leq L_g\vert x-y \vert,
 \]
 	 	for any $x,y\in S$ where  $S$ is a bounded closed set.
 	 	\end{lemma}

\section{Block Proximal Iteratively Reweightly Algorithm with Extrapolation}\label{sec3}
 Inspired by the extrapolation technique \,\cite{Yurii Nesterov} and block update mode  in \cite{XuY}, 
 we propose  
 a block proximal iteratively reweighted algorithm which only requires that each block can be updated at least once within the fixed number of iterations.  For a selected update block $b_k\in \{s_1,\dots,s_m\}$  in iteration $k$,  the specific form is as follows
 \begin{equation}\label{2.3}
 	\begin{cases}
 		{x}^k_{s_i} = {x}^{k-1}_{s_i}, &\text{ if } s_i\ne b_k;\\
 		{x}^k_{s_i}= \arg\min_{ {x}_{s_i}} \{\langle \nabla_{{s_i}} f( {x}^{k-1}_{\ne {s_i}},  {\hat {x}}^k_{s_i}), {x}_{s_i} \rangle  +\frac{1}{2\alpha_k} \Vert  {x}_{s_i} - {\hat{x}}^k_{s_i}\Vert^2 +\sum_{j\in s_i} w^{k-1}_j g( {x}_j) \}  ,&\text{ if } s_i = b_k ,\\
 	\end{cases}
 \end{equation}
 where $\alpha_k$ is the stepsize,  $({x}^{k-1}_{\ne {s_i}},  {\hat {x}}^k_{s_i})$ represents the point $({x}^{k-1}_1,\dots,{x}^{k-1}_{s_i-1},    {\hat {x}}^k_{s_i},{x}^{k-1}_{s_i+1},\dots,
 {x}^{k-1}_{s_m})$, $w^{k-1}_j =\lambda h'(g({x}^{k-1}_j))$ 
 and the extrapolation item is 
 \[  {\hat{x}}^{k}_{s_i} =  {x}^{k-1}_{s_i} + \beta_k ( {x}^{k-1}_{s_i}-  
 {x}^{\rm{prev}}_{s_i}), \]
 where $\beta_k$ is the extrapolation parameter and ${x}^{\rm{prev}}_{s_i}$ is the last updated value of  ${x}^{k-1}_{s_i}$.
  \begin{algorithm}
 	\caption{BPIREe algorithm } 
 	\label{alg1} 
 	\begin{algorithmic}[1] 		
 		\STATE
 		\textbf{Initialization:} $x^{-1}=x^{0}$.
 		\FORALL{$k = 1,2,\dots$}
 		\STATE
 		Choose $b_k\in \{s_1,s_2,\dots,s_m \}$ in a deterministic or random manner. 
 		\STATE
 		${x}^k\gets (\ref{2.3})$  where $\alpha_k, \beta_k $
 		satisfying certain conditions.
 		\IF {stopping criterion is satisfied}
 		\STATE return ${x}^k$.
 		\ENDIF
 		\ENDFOR
 	\end{algorithmic}
 \end{algorithm} 

   The following extra notations are required in this paper since the update can not be periodic. When the algorithm is updated to the $k$th iteration, we use  
  $\mathcal{K}[s_i,k]$ to represent 
  the set of iterations of the block $s_i$, that is,
  \[    \mathcal{K} [s_i,k]: = \{ \kappa: b_\kappa =s_i, 1\leq\kappa \leq k \} \subseteq \{1,\dots,k\},\] and $d^k_{s_i}: =\vert \mathcal{K} [s_i,k] \vert$  to signify the number of the updated iterations for  $s_i$th block. Then
  $\cup_{i=1}^{m} \mathcal {K} [s_i,k] =[k]:=\{1,2,\dots,k\}$  and $\sum_{i=1}^{m} d^k_{s_i} =k$.
  
  After  $k$th iteration, denote ${x}^k$ be the value of ${x}$; and after $j$th update,   ${\tilde{x}}^j_{s_i}$ means the value of ${\tilde{x}}_{s_i}$ for block $s_i$. Letting $j=d^k_{s_i}$, then  ${x}^k_{s_i} =  {\tilde{x}}^j_{s_i}$.
 For $s_i= b_k$, $j=d^k_{s_i}$, the extrapolated point can be indicated as
  \begin{align}\label{2.4}
  	{\hat{x}}^{k}_{s_i} =  {\tilde{x}}^{j-1}_{s_i} +\beta_k (  {\tilde{x}}^{j-1}_{s_i}-  {\tilde{x}}^{j-2}_{s_i}), \quad   0\leq \beta_k \leq 1.
  \end{align}
  The extrapolation coefficient and  Lipschitz constant are divided into $m$ disjoint subsets, which are denoted as
  \begin{align}\label{2.6}
  	\{\beta_\kappa: 1\leq \kappa\leq k\} =\cup_{i=1}^{m} \{\beta_\kappa: \kappa \in\mathcal{K}[s_i,k]  \}: = \cup_{i=1}^{m} \{\tilde \beta^j_{s_i}:  1\leq j\leq d^k_{s_i}\},	
  \end{align}
  \begin{align}\label{2.5}
  	\{L_\kappa: 1\leq \kappa\leq k\} =\cup_{i=1}^{m} \{ L_\kappa: \kappa \in\mathcal{K}[s_i,k]  \}: = \cup_{i=1}^{m} \{ \tilde L^j_{s_i}:  1\leq j\leq d^k_{s_i}\}.
  \end{align}

   For a certain block $s_i$,  the specific forms of 
 the sequence {${x_{s_i}}$}, Lipschitz constant and the extrapolation parameter are denoted as follows
 \begin{align}
 	&\, {\tilde{x}}^{1}_{s_i}, {\tilde{x}}^{2}_{s_i},\dots, {\tilde{x}}^{d^k_{s_i}}_{s_i}\dots;\label {2.7}\\
 	&\,\tilde L^1_{s_{s_i}}, \tilde L^2_{s_i},\dots,\tilde L^{d^k_{s_i}}_{s_i},\dots;\label {2.8}\\
 	&\, \tilde \beta^1_{s_i}, \tilde\beta^2_{s_i},\dots,\tilde \beta^{d^k_{s_i}}_{s_i}.\dots.\label {2.9}
 \end{align}
 
 The stepsize and extrapolation parameter are set as follows
 \begin{align}\label{2.10}
 	\alpha_k = \frac{1}{2L_k},  \quad
 	\tilde \beta^j_{s_i}\leq  \frac{\delta}{6}
 	\sqrt{\frac{\tilde L^{j-1}_{s_i}}{\tilde L^j_{s_i}}},\,   \,\,\delta<1.
 \end{align}
 
 We make the following assumptions, which will be used in the convergence analysis.
 \begin{assumption}\label{assumption2}
 	(Essentially cyclic block update). In Algorithm \ref{alg1}, every block is updated at least once within any $T (T\ge m)$ consecutive iterations.
 \end{assumption}

    \begin{lemma}\label{lemma1}
 	Suppose $\{ x^k\}$ is a sequence generated by  BPIREe algorithm, ${\alpha_k}$ and $\beta_k$ are chosen as in (\ref{2.10}), under the Assumption {\ref{assumption0}}, it has that
 	\begin{align}
 		F(x^{k-1})-F(x^k)&\ge c_1 \tilde L^j_{s_i} \Vert \tilde  x^{j-1}_{s_i}- \tilde x^{j}_{s_i}\Vert^2
 		-c_2 \tilde L^j_{s_i}(\tilde\beta^j_{s_i})^2\Vert\tilde x^{j-2}_{s_i}- \tilde x^{j-1}_{s_i}\Vert^2 \label{7.133}  \\
 		&\ge  c_1 \tilde L^j_{s_i} \Vert \tilde  x^{j-1}_{s_i}- \tilde x^{j}_{s_i}\Vert^2-\frac{c_2\tilde L^{j-1}_{s_i}}{36}\delta^2 \Vert\tilde  x^{j-2}_{s_i}- \tilde x^{j-1}_{s_i}\Vert^2,\label{2.22}
 	\end{align}
 	where $c_1=\frac{1}{4}$, $c_2=9$, $s_i=b_k$ and $j=d^k_{s_i}$.
 \end{lemma}
 \begin{proof}
 	Since $f$ is gradient Lipschitz continuous concerning each block
 	$s_i$ from Assumption \ref{assumption0}(i), it yields that
 	\begin{align}\label{2.2202}
 		f(x^k)\leq f(x^{k-1}) +\langle \nabla_{{s_i}} f(x^{k-1}), x^k_{s_i}-x^{k-1}_{s_i}   \rangle +\frac{L_k}{2} \Vert  x^k_{s_i}-x^{k-1}_{s_i}\Vert^2.
 	\end{align}
 	From the optimality of subproblem (\ref{2.3}), we have
 	\begin{align}\label{2.222}
 		&\langle \nabla_{{s_i}} f(x^{k-1}_{\ne {s_i}}, \hat x^k_{s_i}),x^k_{s_i} \rangle   +\frac{1}{2\alpha_k} \Vert x^k_{s_i} -\hat x^k_{s_i}\Vert^2 +\sum_{j\in s_i} w^{k-1}_{j} g(x^k_{j}) \notag\\
 		\leq  &
 		\langle \nabla_{{s_i}} f(x^{k-1}_{\ne {s_i}}, \hat x^k_{s_i}),x^{k-1}_{s_i} \rangle   +\frac{1}{2\alpha_k} \Vert x^{k-1}_{s_i} -\hat x^k_{s_i}\Vert^2 +\sum_{j\in s_i} w^{k-1}_{j} g(x^{k-1}_{j}),
 	\end{align}
 Notice  $x^{k-1}_{s_j}=x^k_{s_j},\forall j\ne i $.	It then follows that
 	\begin{align*}
 		& F(x^{k-1})- F(x^k) \\
 		= & f(x^{k-1})+\sum_{j\in s_i} h(g(x^{k-1}_j)) -f(x^k) -\sum_{j\in s_i}h(g(x^k_j))  \\
 		\ge& \langle \nabla_{{s_i}} f(x^{k-1}), x^{k-1}_{s_i}-x^{k}_{s_i}  \rangle
 		- \frac{L_k}{2} \Vert  x^k_{s_i}-x^{k-1}_{s_i}\Vert^2+ \sum_{j\in s_i}h'(g(x^{k-1}_{j}))(g(x^{k-1}_{j})-g(x^{k}_{j}))  \\
 		\ge &\langle \nabla_{{s_i}} f(x^{k-1}), x^{k-1}_{s_i}-x^{k}_{s_i}  \rangle  - \frac{L_k}{2} \Vert  x^k_{s_i}-x^{k-1}_{s_i}\Vert^2
 		+\langle \nabla_{{s_i}} f(x^{k-1}_{\ne {s_i}}, \hat x^k_{s_i}), x^k_{s_i}-x^{k-1}_{s_i} \rangle \\  	
 		&+ \frac{1}{2\alpha_k} \Vert  x^k_{s_i}-\hat x^{k}_{s_i}\Vert^2 -
 		\frac{1}{2\alpha_k} \Vert  x^{k-1}_{s_i}-\hat x^{k}_{s_i}\Vert^2\\
 		=&\langle \nabla_{{s_i}} f(x^{k-1})- \nabla_{{s_i}} f(x^{k-1}_{\ne {s_i}}, \hat x^k_{s_i}), x^{k-1}_{s_i}-x^{k}_{s_i} \rangle +\frac{1}{\alpha_k}\langle x^k_{s_i}-x^{k-1}_{s_i},x^{k-1}_{s_i}-\hat x^k_{s_i}   \rangle \\ &+(\frac{1}{2\alpha_k}-\frac{L_k}{2})\Vert x^k_{s_i}-x^{k-1}_{s_i} \Vert^2
 		\\
 		\ge&-\Vert x^{k-1}_{s_i}-x^{k}_{s_i}\Vert(\Vert \nabla_{{s_i}} f(x^{k-1})- \nabla_{{s_i}} f(x^{k-1}_{\ne {s_i}},\hat x^k_{s_i}) \Vert +\frac{1}{\alpha_k} \Vert x^{k-1}_{s_i}-\hat {x}^{k}_{s_i}\Vert)+(\frac{1}{2\alpha_k}-\frac{L_k}{2})\Vert x^k_{s_i}-x^{k-1}_{s_i} \Vert^2   \\
 		\ge& -(\frac{1}{\alpha_k}+L_k)\Vert x^{k-1}_{s_i}-x^{k}_{s_i}\Vert\cdot\Vert x^{k-1}_{s_i}-\hat x^{k}_{s_i}\Vert +(\frac{1}{2\alpha_k}-\frac{L_k}{2})\Vert x^k_{s_i}-x^{k-1}_{s_i} \Vert^2\\
 	 		= & -(\frac{1}{\alpha_k}+L_k)\beta_k\Vert x^{k-1}_{s_i}-x^{k}_{s_i}\Vert\cdot\Vert x^{k-1}_{s_i}-
 	 		\tilde x^{d^{k-1}_{s_i}-1}_{s_i}\Vert +(\frac{1}{2\alpha_k}-\frac{L_k}{2})\Vert x^k_{s_i}-x^{k-1}_{s_i} \Vert^2\\	
 		\ge  &\frac{1}{4}(\frac{1}{\alpha_k}-L_k) \Vert  x^{k-1}_{s_i}-x^{k}_{s_i}\Vert^2
 		-\frac{(\frac{1}{\alpha_k}+L_k)^2}{\frac{1}{\alpha_k}-L_k}\beta^2_k\Vert  x^{k-1}_{s_i}-
 		\tilde x^{d^{k-1}_{s_i}}_{s_i}\Vert^2,	
 	\end{align*}
   where the first inequality derives from (\ref{2.2202}) and the concavity of $h$, the second inequality from (\ref{2.222}), the forth inequality derives from Cauchy-Schwarz inequality
 and the fifth inequality holds due to Young's inequality $ab\leq \frac{1}{\frac{1}{\alpha_k}-L_k}a^2+(\frac{1}{\alpha_k}-L_k)\frac{1}{4}b^2$ for any $a,b\ge0$ with $a =(\frac{1}{\alpha_k}+L_k)\beta_k\Vert x^{k-1}_{s_i}-\tilde x^{d^{k-1}_{s_i}-1}_{s_i}\Vert$ and
 $b= \Vert x^k_{s_i}- x^{k-1}_{s_i}\Vert $. Then the result is obtained by choosing 
   $\alpha_k$ and $\beta_k$  defined as in (\ref{2.10}).
 \end{proof}
\begin{remark}
	When $\alpha_k=\frac{1}{\gamma L_k}$ and $\tilde\beta^{j}_{s_i}\leq\frac{\delta(\gamma-1)}{2(\gamma+1)}\sqrt{ {\tilde L^{j-1}_{s_i}/\tilde L^{j}_{s_i}}}$, $\forall\gamma>1$ and some $\delta<1$. Then (\ref{2.22}) holds with $c_1=\frac{\gamma-1}{4}$ and $c_2=\frac{(\gamma+1)^2}{\gamma-1}$. In a more general case, if $0<\inf_k \alpha_k\leq \sup_k\alpha_k<\infty$, then (\ref{2.22}) holds with $c_1>0$, $c_2>0$ and
	$\tilde \beta^{j}_{s_i}\leq \delta\sqrt{c_1{\tilde L^{j-1}_{s_i}/c_2\tilde L^{j}_{s_i}}}$ for some $\delta<1$.
		\end{remark}
\begin{remark}
 	 Since $d^k_{s_i} =d^{k-1}_{s_i}+1$ for $s_i=b_k$ and $d^k_{s_i} =d^{k-1}_{s_i},\forall {s_i}\ne b_k$, $\sum_{i=p}^{q} a_j=0$ when $q<p$, the inequality (\ref{2.22}) can be rewritten as
 \begin{align}\label{2.2222}
 	F(x^{k+1})-F(x^k)\ge & \sum_{i=1}^{m}\sum_{j=d^{k-1}_{s_i}}^{d^k_{s_i}}		
 	\frac{1}{4}( \tilde L^j_{s_i} \Vert
 	\tilde  x^{j-1}_{s_i}- \tilde x^{j}_{s_i}\Vert^2 -  \tilde L^{j-1}_{s_i}\delta^2 \Vert \tilde  x^{j-2}_{s_i}- \tilde x^{j-1}_{s_i}\Vert^2).
 \end{align}
\end{remark}

\begin{corollary}
		Suppose $\{x^k\}$ is a sequence generated by BPIREe algorithm,  when $f$ is convex with respect to each block of variables 
	and
	$\alpha_k=\frac{1}{L_k}$, 
		  under the Assumption {\ref{assumption0}}, it holds that
	\[F(x^{k-1})- F(x^k)\ge    -\frac{L_k\beta^2_k}{2}\Vert x^{k-1}_{s_i}- x^{d^{k-1}_{s_i}-1}_{s_i} \Vert^2
	+\frac{L_k}{2}\Vert x^k_{s_i}- x^{k-1}_{s_i} \Vert^2.	
	\]
	\end{corollary}
\begin{proof}
From (\ref{2.2202}), we get
	\begin{align*}
		& F(x^{k-1})- F(x^k) \\
		= &f(x^{k-1}) -f(x^{k})+ \sum_{j\in s_i}h(g(x^{k-1}_j))	-\sum_{j\in s_i}h(g(x^k_j))\\
		\ge& f(x^{k-1})- f(x^{k-1}_{\ne{s_i}},\hat x^k_{s_i})-
		\langle \nabla f(x^{k-1}_{\ne{s_i}},\hat x^k_{s_i}), x^k_{s_i}-\hat x^k_{s_i}\rangle    -\frac{L_k}{2}\Vert x^k_{s_i}-\hat x^k_{s_i} \Vert^2  \\
	   &\,+\lambda\sum_{j\in s_i}h'(g(x^{k-1}_j))(g(x^{k-1}_j)-g(x^k_j)) \\ 	
		\ge& f(x^{k-1})-f(x^{k-1}_{\ne{s_i}},\hat x^k_{s_i})-
		\langle \nabla f(x^{k-1}_{\ne{s_i}},\hat x^k_{s_i}), x^k_{s_i}-\hat x^k_{s_i}\rangle -\frac{L_k}{2}\Vert x^k_{s_i}-\hat x^k_{s_i} \Vert^2\\
		&+\langle \nabla f(x^{k-1}_{\ne{s_i}},\hat x^k_{s_i}), x^k_{s_i}-  x^{k-1}_{s_i}\rangle+\frac{L_k}{2}\Vert x^k_{s_i}-\hat x^k_{s_i} \Vert^2-\frac{L_k}{2}\Vert x^{k-1}_{s_i}-\hat x^k_{s_i} \Vert^2+
		\frac{L_k}{2}\Vert x^k_{s_i}- x^{k-1}_{s_i} \Vert^2 \\
		=&f(x^{k-1})-f(x^{k-1}_{\ne{s_i}},\hat x^k_{s_i})- \langle \nabla f(x^{k-1}_{\ne{s_i}},\hat x^k_{s_i}), x^{k-1}_{s_i}-  \hat x^{k}_{s_i}\rangle- \frac{L_k}{2}\Vert x^{k-1}_{s_i}-\hat x^k_{s_i} \Vert^2  \\
	& \,+\frac{L_k}{2}\Vert x^k_{s_i}- x^{k-1}_{s_i} \Vert^2 \\
		\ge &- \frac{L_k}{2}\Vert x^{k-1}_{s_i}-\hat x^k_{s_i} \Vert^2+
		\frac{L_k}{2}\Vert x^k_{s_i}- x^{k-1}_{s_i} \Vert^2\\
		= &-\frac{L_k\beta^2_k}{2}\Vert x^{k-1}_{s_i}- x^{d^{k-1}_{s_i}-1}_{s_i} \Vert^2
		+\frac{L_k}{2}\Vert x^k_{s_i}- x^{k-1}_{s_i} \Vert^2,	 	  	
	\end{align*}
	where the first inequality derives from (\ref{2.2202}) and the concavity of $h$, the second inequality holds due to the strong convexity of the subproblem (\ref{2.3}) and the third inequality  comes from the convexity of $f$.
	This completes the proof.
	\end{proof}
\begin{remark}
 	It is observed that when $f$ is convex concerning each block of variables, the inequality  (\ref{7.133}) holds with  $c_1=c_2=\frac{1}{2}$ and the extrapolation parameter can be chosen as 
 $\beta^2_k\leq   {(\tilde L^{j-1}_{s_i}\delta^2)/\tilde L^j_{s_i} }\,(\delta<1)$.
\end{remark}
\begin{assumption}\label{assumption1}
For any $k\in\mathbb{N}$,	the extrapolation parameter $\beta_k$ is chosen to satisfy
	$F(x^k) \leq F(x^{k-1})$.
\end{assumption}
 Now we make an illustration   
 of the Assumption \ref{assumption1} is rational. 
\begin{lemma}
	For any $s_i =b_k$, suppose
 \begin{equation}\label{2.22222}
	x^{k-1}_{s_i} \notin \arg\min_{{s_i}} \{\langle \nabla_{{s_i}} f(x^{k-1}), x_{s_i}-x^{k-1}_{s_i} \rangle  +\frac{1}{2\alpha_k} \Vert x_{s_i} - x^{k-1}_{s_i}\Vert^2 +\sum_{j\in s_i} w^{k-1}_{j} g(x_{j}) \},
 \end{equation}
then there  exists $\bar\beta_k>0$ such that $F(x^k) \leq F(x^{k-1})$ for any $\beta_k \in [0, \bar \beta_k]$.
\end{lemma}

  \begin{proof}
 	Since the subproblem (\ref{2.22222}) is strongly convex, the
 	 proximal operator of
 	$\alpha_k\sum_{j\in s_i} w^{k-1}_{j} g(x_{j})$ signal valued.
 Then we know its 
 	proximal operator 
 	is continuous	
 	according to the Corollary 5.20   and Exercise\,5.23 in \cite{Rockafellar}.
 	Let \,$\hat x^k_i(\beta)$ be the extrapolation  with parameter 
 	$\beta$, the subproblem is rewritten as 
 	\[  x^k_{s_i}(\beta)={\rm{prox}}_{\alpha_k\sum_{j\in s_i} w^{k-1}_{j} g(x_{j})}
 	(\hat x^k_{s_i}(\beta)-\alpha_k\nabla_{{s_i}}f(x^{k-1}_{\ne {s_i}},\hat x^k_{s_i}(\beta))).\]
 	By\,(\ref{2.2222}) and 
 	(\ref{2.22222}), we have
 	\begin{align}\label{8.15}
 		F(x^{k-1})-F(x^k(0))\ge\Vert x^{k-1}-x^k(0) \Vert^2>0.
 	\end{align}
 	From the subproblem \,(\ref{2.3}), we get
 	\begin{align*}
 		&\langle \nabla_{{s_i}} f({x}^{k-1}_{\ne {s_i}},
 		{\hat {x}}^k_{s_i}(\beta)), {x}^k_{s_i}(\beta)-{\hat x}^k_{s_i}(\beta) \rangle
 		+\frac{1}{2\alpha_k} \Vert  {x}^k_{s_i}(\beta) -
 		{\hat{x}}^k_{s_i}(\beta)\Vert^2 +\sum_{j\in s_i} w^{k-1}_j g( {x}_j(\beta))  \\
 		&\leq \langle \nabla_{{s_i}} f({x}^{k-1}_{\ne {s_i}},
 		{\hat {x}}^k_{s_i}(\beta)), {x}_{s_i}-{\hat x}^k_{s_i}(\beta)\rangle
 		+\frac{1}{2\alpha_k} \Vert  {x}_{s_i} -
 		{\hat{x}}^k_{s_i}(\beta)\Vert^2 +\sum_{j\in s_i} w^{k-1}_j g({x}_j).
 	\end{align*}
 	Let\,$\beta\to\,0^+$, it holds that
 	\[ \lim_{\beta\to\,0^+} \sup  w^{k-1}_j g( {x}_j(\beta))\leq
 	w^{k-1}_j g({x}_j(0)).
 	\]
 	By virtue of the closedness of\,$g$, we have
 	\[ \lim_{\beta\to\,0^+}\inf   w^{k-1}_j g( {x}_j(\beta))\ge
 	w^{k-1}_j g({x}_j(0)),
 	\]
 	and then \[ \lim_{\beta\to\,0^+}  w^{k-1}_j g( {x}_j(\beta))=
 	w^{k-1}_j g({x}_j(0)).
 	\]
 	Since $h$  is continuous
 	and is increasing,  \,$\{g(x^k)\}$ is bounded for any \,$k\in{\mathcal{K}_{s_i}}$,  there exists \,$c_0>0$
 	satisfying\,$w^{k}_{s_i}\ge c_0 $. Hence, we have
 	\[ \lim_{\beta\to\,0^+}   g( {x}_j(\beta))=
 	g({x}_j(0)).
 	\]
 	Further, by the continuous differentiability of \,$g$, we infer that
 	\,$\lim_{\beta\to\,0^+}  h(g({x}_j(\beta)))=
 	h(g({x}_j(0))).$
 	Together with the continuity of \,$f$, it holds that
 	$ \lim_{\beta\to\,0^+}   F( {x}_j(\beta))=
 	F({x}_j(0)).$  
 According to \,(\ref{8.15}), then there exists\,$\bar\beta_k>0$ satisfies
 	$F(x^{k-1})-F(x^k(\beta))\ge 0$, $\forall\beta\in[0,\bar\beta_k]$. This completes the proof.
 \end{proof}

  \begin{lemma}\label{prop00}
 	Suppose $\{x^k\}$ is the sequence generated by BPIREe algorithm, under the Assumptions \ref{assumption0}-\ref{assumption1}, when $\alpha_k$ and $\beta_k$ are chosen as (\ref{2.10}), it holds that
 	\begin{enumerate}[\rm(i)]
 		\item the sequence $\{x^k\}$ is bounded.
 		\item  $\label{2.001}
 		\sum_{k=0}^{\infty} \Vert x^{k+1}-x^k \Vert^2<\infty.$	 	
 	\end{enumerate} 	
 \end{lemma}
 \begin{proof}
 	(i) Since $F(x^k)$ is nonincreasing and satisfies  $F(x^k)\leq F(x^0)$ for any $k\in\mathbb{N}$,  
 	the sequence $\{x^k\}$ is bounded from the coerciveness of $F$. 
 	
 	(ii) The proof is similar to  Proposition 1 \cite{XuY} 
 	and can be easily got.
 \end{proof}
 \begin{lemma}\label{prop0}
 	Suppose the Assumptions \ref{assumption0}-\ref{assumption1} hold,
 	$\{x^k\}$ is the sequence generated by  BPIREe algorithm  for a certain iteration $k\ge 3T$, assume $x^\kappa\in B_{\rho}(x^*):=\{x:\Vert x- x^* \Vert<\rho \}$, $\kappa =k-3T, k-3T+1,\dots,k$ for some $x^*$  and $\rho>0$, for each block $s_i$, $\nabla_{{s_i}}f(x)$ is Lipschitz continuous with Lipschitz modulus $L_G$ within $B_{4\rho} (x^*)$ concerning $x$, i.e.,
 	\begin{align}\label{2.02}
 		\Vert\nabla_{{s_i}}f(y)-\nabla_{{s_i}}f(z)\Vert\leq L_G\Vert y-z \Vert, \forall y,z\in B_{4\rho} (x^*),
 	\end{align}
 	then
 	\begin{align}\label{2.03}
 		dist(0, \partial F(x^k))\leq c_3\sum_{i=1}^{m}\sum_{j=d^{k-3T}_{s_i}+1}^{d^{k}_{s_i}} \Vert \tilde x^{j-1}_{s_i} -\tilde x^{j}_{s_i} \Vert,
 	\end{align}
 	where $c_3 =2(2L\frac{\pi}{\delta}+L_G)+\frac{L_0L_hL_g}{\delta}.$
 \end{lemma}
 \begin{proof}
 	When  the  $s_i$th block of $x$ is updated to $x^k_{s_i}$,  the preceding value of the $s_i$th block is denoted as  $u_{s_i}$, the $s_j$th  block $(s_j\ne s_i)$ value is represented as $y_{\ne s_i}^{({s_i})}$, the extrapolated point of the $s_i$th block is denoted as $z_{s_i}$, and the Lipschitz constant of
 	$\nabla_{{s_i}}f(y_{\ne s_i}^{(s_i)},x_{s_i})$ with regard to
 	$x_{s_i}$ is expressed as $\tilde L_{s_i}$, then we have
 	\[ x^k_{s_i}= \arg\min_{x_{s_i}} \{\langle \nabla_{x_{s_i}}f(y_{\ne s_i}^{(s_i)},z_{s_i}), x_{s_i}-z_{s_i}    \rangle+ \tilde L_{s_i} \Vert x_{s_i}-z_{s_i}\Vert^2 +\sum_{j\in s_i}w_{j}g(x_{j})    \}, \]
 	where $w_j =\lambda h'(g(u_j))$.
 	
 	By the optimality condition of the above subproblem,  we have 
 	\begin{align}\label{2.030}
 		\nabla_{{j}}f(y_{\ne s_i}^{({s_i})},z_{s_i})+2 \tilde L_j  (x^k_{j}-z_{j})+w_{j}v^k_{j}=0,
 	\end{align}
 	where $j\in s_i, w_j=h'(g(u_j)),\,v^k_{j} \in\partial g(x^k_{j}) $.
 	Besides, from the optimality condition of (\ref{1.1}), we have
 	\begin{align}\label{2.031}
 		\nabla_{{s_i}} f(x^{k})+  W^{k}_{s_i}v_{s_i}^{k}\in\partial_{{s_i}} F(x^{k}),
 	\end{align}
 	where $ W^k_{s_i}={\rm{Diag}}(w^k_j)$, $w^k_j=h'(g(x^k_j))$.
	 	
 	According to Assumption \ref{assumption2},
 the value $x^k_{s_i}$ may   be obtained at  some earlier iteration but not at the $k$th iteration,    which must be  between $k-T$ and $k$. 
 Besides, for any pair $(s_i,s_j)$, there exists $\kappa_{s_i,s_j}$ between $k-2T$ and $k$ satisfying
 $y^{s_i}_{\ne s_i}=x^{\kappa_{s_i,s_j}}_{s_j}$ and for each block $s_i$,	it holds that 
 \[ z_{s_i}=x^{\kappa^{s_i}_2}_{s_i}+\tilde \beta_{s_i}(x^{\kappa^{s_i}_2}_{s_i}-x^{\kappa^{s_i}_1}_{s_i}),\]
 where $k-3T\leq \kappa^{s_i}_1\leq \kappa^{s_i}_2\leq k$.			 
 Denote
 \[\xi^{k}_{s_i} : =- W^k_{s_i}(W_{s_i})^{-1} (\nabla_{{s_i}}f(y_{\ne s_i}^{( {s_i})},z_{s_i})+2 \tilde L_{s_i} (x^k_{s_i}-z_{s_i}))+\nabla_{s_i}f(x^k).\] 
 In view of (\ref{2.030}) and (\ref{2.031}), it obviously has that $\xi^{k}_{s_i}\in \partial_{{s_i}}F(x^k).$

 Since $\{x^k\}$ is bounded according to Lemma \ref{prop00} and $\nabla f$ is continuous, there exists $L_0>0$ such that
$\Vert \nabla_{{s_i}}f(y_{\ne s_i}^{( {s_i})},z_{s_i})\Vert \leq L_0.$
It is observed that $h'$ is nozero, continuous and  $g(x^k)$ is bounded from Assumption \ref{assumption0},  then there exists $\delta,\pi>0$  satisfying
$\delta \leq h'(g(x^k_{s_i}))\leq \pi$ and  
\[  \Vert W^k_{s_i} \Vert_2\leq \max\frac{1}{ h'(g(x^k_{s_i}))}\leq \frac{1}{\delta},\quad \Vert W^{k+1}_{s_j} W^{-1}_{s_j} \Vert_2\leq \max\frac{h'(g(x^{k+1}_{s_i})) }{h'(g(u_{s_i}))}\leq \frac{\pi}{\delta}.
\] 
From Assumption \ref{assumption0}  and Lemma \ref{lemma 2.4}, for any $j\in s_i$, 
 we have 
\begin{align*}
	\vert(W^k_{s_i}-W_{s_i})_{j,j} \vert=\vert  h'(g(x^k_{j}))-h'(g(u_{j}))\vert\leq L_h\vert g(x^k_{j})- g(u_{j}) \vert \leq L_hL_g\vert x^k_{j}-u_{j} \vert .
\end{align*}
Then it follows that
	\begin{align*}
		\Vert \xi^k_{s_i}\Vert
		& \leq	2\tilde L_{s_i}\Vert W^k_{s_i}(W_{s_i})^{-1}\Vert\cdot \Vert x^k_{s_i}-z_{s_i}\Vert  
		   +\Vert W^k_{s_i}(W_{s_i})^{-1} \nabla_{{s_i}}f(y_{\ne s_i}^{( {s_i})},z_{s_i})+\nabla_{{s_i}}f(x^k)\Vert	\\	
			& \leq 	 2\tilde L_{s_i}\frac{\pi}{\delta}\Vert x^k_{s_i}-z_{s_i}\Vert			
			+\Vert W^k_{s_i}(W_{s_i})^{-1} \nabla_{{s_i}}f(y_{\ne s_i}^{( {s_i})},z_{s_i})+ \nabla_{{s_i}}f(y_{\ne s_i}^{( {s_i})},z_{s_i}) \Vert\\		 
			  &\quad +\Vert\nabla_{{s_i}}f(x^k)-\nabla_{{s_i}}f(y_{\ne s_i}^{( {s_i})},z_{s_i}) \Vert\\	
	& \leq 2\tilde L_{s_i}\frac{\pi}{\delta}\Vert x^k_{s_i}-z_{s_i}\Vert+\frac{L_0}{\delta}\Vert W^k_{s_i}-W_{s_i} \Vert+L_G\Vert x^k-(y_{\ne s_i}^{( {s_i})},z_{s_i}) \Vert\\	
	&\leq(2L\frac{\pi}{\delta}+L_G)\Vert x^k_{s_i}-z_{s_i}\Vert+\frac{L_0L_hL_g}{\delta}\Vert x^k_{s_i}-u_{s_i}\Vert+L_G\sum_{j\ne i}\Vert x^k_{s_j}-y_{\ne s_i}^{( {s_i})} \Vert.
	\end{align*}
	Therefore,
	\begin{align*}
		dist(0,\partial F(x^{k}))
		& =\sqrt{\sum_{i=1}^{m} dist^2(0,\partial F(x^{k})_{s_i}) }  \\
		& \leq  \sum_{i=1}^{m} dist(0,\partial F(x^{k})_{s_i}) \\
		&\leq \sum_{i=1}^{m}  \Vert \xi^k_{s_i}\Vert\\
		&\leq c_3\sum_{i=1}^{m} \sum_{j=d_i^{{k-3T}}+1}^{d_i^k} \Vert 
		\tilde x_i^{j-1}-\tilde x_i^j \Vert,	
	\end{align*}
where $c_3 =2(2L\frac{\pi}{\delta}+L_G)+\frac{L_0L_hL_g}{\delta}.$
\end{proof}
 \begin{lemma} (Subsequence convergence) \label{prop 2}
 	Suppose $\{x^k\}$ is the sequence generated by BPIREe algorithm,   $\alpha_k$, $\beta_k $ are chosen as (\ref{2.10}), under the Assumptions 
 	\ref{assumption0}-\ref{assumption1}, 
 	then
 	\begin{enumerate}[\rm(i)]
 	\item there exists a subsequence $\{x^{k}\}_{k\in\mathcal{K}}$ of $\{x^k\}$ such that
 	\[  \lim_{k\in\mathcal{K},k\to +\infty} F(x^{k}) =F(x^*),
 	\]	
 	where $x^*$ is any limit point of $\{x^k\}$.
 	\item  any limit point ${x^*}$ of $\{x^k \}$ is a critial point of problem (\ref{1.1}).
 	\end{enumerate}
 \end{lemma}
 \begin{proof}(i) 	 
 	  Since $\{x^k\}$ is bounded from Lemma \ref{prop00}(i),  it has an accumulation point which might as well denote as $x^*$. Then there exists an index set $\mathcal{K}$ such that  $\lim_{k\in\mathcal{K},k\to+\infty,} {x^k}= x^*.$ 
 	   From Lemma \ref{prop00} (ii),  it has that $\lim_{k\to+\infty}\Vert x^{k+1}-x^k \Vert= 0$
 	  and then 
   $\lim_{k\in\mathcal{K}, k\to+\infty} {x^{k+\kappa}} =\bar x$   $\kappa\in\mathbb{N}$, let
 	\begin{align*}
 		\mathcal {K}_{s_i} = \{k\in  \cup_{\kappa=0}^{T-1} (\mathcal {K}+ \kappa):b_k ={s_i}\}, i=1,2,\dots,m.
 	\end{align*}
 	Then $\mathcal{K}_{s_i}$ is an infinite set. Since $L_k$ is bounded, without loss of generality, we assume $L_k\to\tilde L_{s_i}$ with $k\in\mathcal{K}_{s_i}$ as $k\to+\infty$.  
 	From $\alpha_k = \frac{1}{2L_k}$ for any $k\in\mathcal{K}_{s_i}$ and  by the subproblem (\ref{2.3}), we deduce
 	\begin{align*}
 		&\langle \nabla_{x_{s_i}}f(x^{k-1}_{\ne {s_i}}, \hat x^k_{{s_i}}), x^k_{s_i} -\hat x^k_{s_i} \rangle + L_k \Vert  x^{k+1}_{s_i} -\hat x^k_{s_i}\Vert^2
 		+ w^k_{s_i} g(x^k_{s_i}) \\
 		&\leq \langle \nabla_{x_{s_i}}f(x^{k-1}_{\ne {s_i}}, \hat x^k_{{s_i}}),  x^*_{s_i} -\hat x^k_{s_i} \rangle + L_k \Vert x^*_{s_i} -\hat x^k_{s_i}\Vert^2+ w^k_{s_i} g( x^*_{s_i}).
 	\end{align*}
 Since $h$ is increasing, continuous and $\{g(x^k)\}$ is bounded,  for any
 $k\in{\mathcal{K}_{s_i}}$, there exists $c_0>0$ such that
 $w^{k}_{s_i}\ge c_0 $. Taking limits on both sides of the above inequality, we infer
 that
 	$\lim_{k\to+\infty,k\in {\mathcal{K}_{s_i}}}\sup g(x^k_{s_i})\leq g(x^*)$.
 	By the closedness of $g$,  we get $\lim_{k\to+\infty,k\in {\mathcal{K}_{s_i}}} g(x^k_{s_i}) = g(x^*)$. Together with the continuity of
 	 $f$, we have $\lim_{k\to+\infty,k\in {\mathcal{K}_{s_i}}} F(x^{k}) = F(x^*)$. Then it holds that $\lim_{k\to+\infty,k\in {\mathcal{K}}} F(x^{k}) = F(x^*)$.   
 	 	
 	 (ii)	
 	  For any limit point $x^*$, there exists a subsequence $\{x^k\}_{k\in\mathcal{K}}$ of
 	  $\{x^k\}_{k\in\mathbb{N}}$
 	   such that $ x^k_{k\in\mathcal {K}}\to x^*$ as $k\to+\infty$. According to Lemma \ref{prop0} and
 	     Lemma \ref{prop00}(ii), then $\xi^{k}_{k\in\mathcal{K}}\to 0$ for any $\xi^k \in\partial F(x^k)$. From  the closedness of $\partial F$  and $\lim_{k\to+\infty,k\in{\mathcal{K}}} F(x^{k}) = F(x^*)$,  we have $0\in \partial F(x^*)$, this completes the proof. 
 \end{proof}
  On the basis of
 the precious Lemmas, we get global convergence  and  local convergence rate of the sequence generated by BPIREe algorithm by using the KL property. The proofs  are similar to Theorem 2, Theorem 3 in \cite{XuY} and we omit them here.
 \begin{theorem}\label{theorem1}
 	Suppose $\{x^k \}$ is the sequence generated by the BPIREe algorithm, 
 	$F$ satisfies the $KL$ property around the limit point $x^*$ and for each block $s_i$, $\nabla_{{s_i}}f(x)$ is Lipschitz continuous within $B_{4\rho}(x^*)$ with respect to $x$, under the Assumptions 
 	\ref{assumption0}-\ref{assumption1}, then
 	\[\lim_{k\to+\infty} x^k =x^*. \]
 \end{theorem}
 \begin{theorem}\label{theorem2}
 	Under the condition of Theorem \ref{theorem1}, the concave function is chosen as  $\phi(s)=cs^{1-\theta}(c>0)$, then it holds that
 	\begin{enumerate}[\rm(i)]
\item
if $\theta = 0$, then
there exists $k_0\in\mathbb{N}$ satisfying $x^k\equiv x^*$ for any $k\ge k_0;$ 
 \item if $\theta\in(0,\frac{1}{2}]$, $\Vert x^k-x^* \Vert \leq c'_4\alpha^k$, $\forall k\in\mathbb{N}$, for a certain $c'_4>0$,  $\alpha\in[0,1);$	 
 \item if $\theta \in (\frac{1}{2},1)$, $\Vert x^k -x^*\Vert \leq c_4k^\frac{-(1-\theta)}{2\theta-1}$, $\forall k\in\mathbb{N}$, for a certain $c_4>0$.
	\end{enumerate}
\end{theorem}

 \section{$\ell_p$ regularization problem}\label{sec6}
 In this section, we focus on the special case of (\ref{1.1}) 
 with $h(u)=(u+ \varepsilon^2_i)^p (0<p<1)$, 
 $g(u)=\vert u\vert$, $u\in\mathbb{R}$, $\varepsilon=(\varepsilon_1,\varepsilon_2,\dots,\varepsilon_n)\in
 \mathbb{R}^n$,
  i.e., 
 \begin{align}\label{4.0}
 	\mathop{\min}_{{x},\varepsilon\in\mathbb{R}^n} \quad F(x,\varepsilon):=f({ {x_{s_1}}},{ {x_{s_2}}},\dots,{{x_{s_m}}})+\lambda \sum_{i=1}^{m}\sum_{j\in{{{s_i}}}} (\vert x_i\vert+\varepsilon^2_i)^p.
 \end{align}  
 The  model (\ref{4.0}) can be seen as a smoothing approximation of the following $\ell_p$ regularization problem 
 \begin{align}\label{4.1}
 	\mathop{\min}_{ {x}\in\mathbb{R}^n} \quad F({ {x_{s_1}}},{ {x_{s_2}}},\dots,{ {x_{s_m}}}):=f({ {x_{s_1}}},{ {x_{s_2}}},\dots,{{x_{s_m}}})+\lambda \sum_{i=1}^{m}\sum_{j\in{{{s_i}}}} \vert x_j \vert^p.
 \end{align}  

  Since $\ell_p (0<p<1) $  regularization problem is non-Lipschitz continuous,  some researchers focus on solving its  smoothing approximation (\ref{4.0}) by skillfully updating its smoothing factor \cite{Chen X, LuZ}. Recently, Wang et al. \cite{Wang H2} developed an iteratively reweighted  $\ell_1$ algorithm
 by   
 using an adaptively updating strategy  (\ref{4.2b}) 
    for sovling $\ell_p$ regularization problem. When for sufficiently
  large iterations, the iterates stay in the same orthant,  zero components are fixed
  and the nonzero components are kept away from zero.
   The algorithm acts to solve a smooth problem in the reduced space. 
  Inspired by this, 
  we propose a block proximal iteratively reweighted algorithm  with extrapolation for solving the $\ell_p$ regularization problem (\ref{4.1}) by using the effective update strategy of the smoothing factor.
  The proposed algorithm is abbreviated as the BPIREe-$\ell_p$ algorithm.
  The subproblem in (\ref{2.3}) is specially developed as follows 
   \begin{equation}\label{4.2a}
  	\begin{cases}
  		{x}^k_{s_i} = {x}^{k-1}_{s_i}, &\text{ if } s_i\ne b_k,\\
  		{x}^k_{s_i}= \arg\min_{ {x}_{s_i}} \{\langle \nabla_{{s_i}} f( {x}^{k-1}_{\ne {s_i}},  {\hat {x}}^k_{s_i}), {x}_{s_i} \rangle  +\frac{1}{2\alpha_k} \Vert  {x}_{s_i} - {\hat{x}}^k_{s_i}\Vert^2 +\sum_{j\in s_i} w^{k-1}_j\vert {x}_j \vert \}  ,&\text{ if } s_i = b_k ,\\
  	\end{cases}
  \end{equation}
  where $w^{k-1}_j=p(\vert x^k_j\vert +(\varepsilon^k_j)^2)^{p-1}$ and the  update mode of smoothing factor $\varepsilon^k_j$ takes the following form
 \begin{equation}\label{4.2b}
 	\begin{cases}
 		\varepsilon^{k+1}_i=\varepsilon^k_i, &{\rm if}\, x^{k+1}=0,\\
 		\varepsilon^{k+1}_i\leq \sqrt{\mu}\varepsilon^k_i,
 		&{\rm if}\, x^{k+1}\ne0,	
 	\end{cases}
 \end{equation} 
where $\mu\in(0,1)$.
Some special properties of  the the BPIREe-$\ell_p$ algorithm  are obtained given below which is critical to the convergence analysis of the proposed algorithm.  
  The proof is similar to theorem\,$1$ of \cite{Wang H}. We will not go into the details here. 
 
\begin{theorem}\label{theorem4.1}
 Suppose $\{x^k\}$ is the sequence generated by BPRIEe-$\ell_p$ algorithm, under the Assumption \ref{assumption0}, then there exists $c_5>0$ and $K\in\mathbb{N}$ such that  
  \begin{enumerate}[\rm(i)]
  	\item if $w^k_i>c_5/\lambda$, then $x^k_i\equiv 0$ and $\varepsilon^k_i\equiv \varepsilon^K_i $ for all 
  	$k>K;$
  	\item the index sets 
  	$\mathcal{I}(x^k)$ and $\mathcal{A}(x^k)$ are fixed for all $k>K$. Hence, 
  	we can denote $\mathcal{I}^*=\mathcal{I}(x^k)$ and $\mathcal{A}^*=\mathcal{A}(x^k)$ for any $k>K$, where  $\mathcal{I}(x)=\{i|x_i\ne 0,i=1,2,\dots,n\}$ and $\mathcal{A}(x)$ is the complement of $\mathcal{I}(x)$, i.e., $\mathcal{A}(x)=\{i|x_i= 0, i=1,2,\dots,n\};$
  	\item for each $i\in\mathcal{I}^*$ and any $k>K$,
  	$\vert x^k_i\vert\ge(\frac{c_5}{p\lambda})^{\frac{1}{p-1}}-(\varepsilon^k_i)^2${\rm ;}
   \item for any limit point $x^*$ of $\{x^k\}$, it holds that $\mathcal{I}(x^*)=\mathcal{I}^*$, $\mathcal{A}(x^*)=\mathcal{A}^*$ and 
   $\vert x^*_i\vert\ge(\frac{c_5}{p\lambda})^{\frac{1}{p-1}}, i\in \mathcal{I}^*${\rm ;}
    \item there exists $s\in\{-1,0,1\}^n$ such that 
    $sign(x^k)\equiv s$ for any $k>K${\rm ;}
    \item for all $k>K$,
     $\varepsilon^k_i\equiv\varepsilon^K_i>0$, $i\in\mathcal{A}^*$
    and $\varepsilon^k_i\to 0$ as $k\to+\infty$, $i\in\mathcal{I}^*$.
  	\end{enumerate}  	
\end{theorem}
According to the Theorem \ref{theorem4.1}, we know  that the BPIREe-$\ell_p$ algorithm  does  not need the differentiability of the regularizer and the smoothing factor bounded away from 0 when $k$ is sufficiently large. Moreover, 
we get that $\{x^k_{\mathcal{I}^*}\}$ keeps in the interior of the same orthant of $\mathbb{R}^{\mathcal{I}^*}$
  after the $k$th iteration, then $F$ turn into a function of $\{x_{\mathcal{I}^*},\varepsilon_ {\mathcal{I}^*}\}$ for sufficiently large $k$, i.e., 
$(x^*_{\mathcal{I}^*},0_{\mathcal{I}^*})
\mapsto  F(x^*_{\mathcal{I}^*},0_{\mathcal{I}^*})$ 
\begin{align*} 
	F(x_{\mathcal{I}^*},\varepsilon_{\mathcal{I}^*}):=f({{x_{{\mathcal{I}^*}_{s_1}}}},{ {x_{{\mathcal{I}^*}_{s_2}}}},\dots,{{x_{{\mathcal{I}^*}_{s_m}}}})+ \sum_{i=1}^{m}\sum_{j\in{{{{\mathcal{I}^*}_{s_1}}}}} (\vert x_i\vert+\varepsilon^2_i)^p.
\end{align*}  
 Therefore, we assume that the reduced function $F(x_{\mathcal{I}^*}, \varepsilon_{\mathcal{I}^*})$ has the KL property at   $(x^*_{\mathcal{I}^*}, 0_{\mathcal{I}^*})\in\mathbb{R}^{2\vert {\mathcal{I}^*}\vert}$.
 \begin{assumption}\label{assumption_4}
 Suppose the reduced function $(x_{\mathcal{I}^*},\varepsilon_{\mathcal{I}^*})
 \mapsto     F(x_{\mathcal{I}^*},\varepsilon_{\mathcal{I}^*})$ 
 has the KL property at  every $(x^*_{\mathcal{I}^*}, 0_{\mathcal{I}^*})\in\mathbb{R}^{2\vert {\mathcal{I}^*}\vert}$,  where $x^*$ is the limit point of the sequence generated by BPIREe-$\ell_p$.
 \end{assumption}

  For simplicity, we suppose ${\mathcal{I}^*}= \{1,2,\dots,n\}$  in the following convergence analysis.
 It is not difficult to find that the results in Lemmas \ref{lemma1}-\ref{prop00} and Lemma \ref{prop 2} hold concerning the 
 objective function $F(x^k,\varepsilon^k),\,k\in\mathbb{N}$. Now, we give some required lemmas before establishing the entire convergence of the sequence generated by the BPIREe-$\ell_p$ algorithm.

\begin{lemma}\label{prop4.0}
	Suppose $\{x^k\}$ is the sequence generated by  BPIREe-$\ell_p$ algorithm, for a certain  iteration $k\ge 3T$, assume $x^\kappa\in B_{\rho}(x^*):=\{x:\Vert x-x^* \Vert<\rho \}$, $\kappa =k-3T, k-3T+1,\dots,k$ for some $x^*$  and $\rho>0$, for each block $s_i$, $\nabla_{{s_i}}f(x)$ is Lipschitz continuous with Lipschitz modulus $L_G$ within $B_{4\rho} (x^*)$ with respect to $x$, i.e., the inequality (\ref{2.02})
	holds,
	then
	\begin{align}\label{4.03}
		\Vert\nabla F(x^k,\varepsilon^k) \Vert\leq c_6\sum_{i=1}^{m} 	
		 \sum_{j=d^{k-3T}_{s_i}+1}^{d^{k}_{s_i}} (\Vert \tilde x^{j-1}_{s_i} -\tilde x^{j}_{s_i} \Vert  +\Vert\varepsilon^{j-1}_{s_i}\Vert_1-\Vert\varepsilon^{j}_{s_i}\Vert_1),
	\end{align}
	where $c_6 =c_3+\frac{2c_5\sqrt{\mu}}{1-\sqrt{\mu}}.$
\end{lemma}
\begin{proof}
	According to  Lemma \ref{prop0}, 
	it holds that 
\begin{align}\label{4.7}
\Vert\nabla_{x}F(x^k,\varepsilon^k)\Vert 
\leq c_3\sum_{i=1}^{m}\sum_{j=d^{k-3T}_{s_i}+1}^{d^{k}_{s_i}} \Vert \tilde x^{j-1}_{s_i} -\tilde x^{j}_{s_i} \Vert, 
\end{align}
and from the proof of Lemma 2 in
\cite{Wang H},  we have  
\begin{align}\label{4.8}
\Vert\nabla_{\varepsilon} F(x^k,\varepsilon^k)  \Vert 
&\leq 
\frac{2c_5\sqrt{\mu}}{1-\sqrt{\mu}}
\sum_{i=1}^{m}
(\Vert \varepsilon^{k-1}_{s_i} \Vert_1-
\Vert \varepsilon^{k}_{s_i} \Vert_1
).
\end{align}
By combining (\ref{4.7}) and (\ref{4.8}),  we get the desired result.  
\end{proof}

\begin{lemma}\label{lemma4.1}
	Given nonnegative sequences $\{A_{i,j}\}_{j\ge 0}$, 
	$\{\alpha_{i,j}\}_{j\ge0}$, $\{C_{i,j}\}_{j\ge0}$, $i=1,\dots,m$ and $\{B_t\}_{t\ge 0}$ where $\{C_{i,j}\}$ is
	a nonincreasing sequence and $\lim_{j\to+\infty}C_{i,j}=0$, 
	 suppose 
 \[ 0<\underline\alpha=\inf_{i,j} \alpha_{i,j}\leq \sup_{i,j}\alpha_{i,j}=\bar\alpha_{i,j}<\infty,
	\] 
and
\begin{align}\label{4.4}
\sum_{i=1}^{m}\sum_{j=n_{i,t}+1}^{n_{i,t+1}}
(\alpha_{i,j}A^2_{i,j}-\alpha_{i,j-1} w^2A^2_{i,j-1})\leq B_t\sum_{i=1}^m\sum_{j=n_{i,t-1}+1}^{n_{i,t}} (A_{i,j}+C_{i,j-1}-C_{i,j}), \,0\leq t\leq M,
\end{align}
where $0\leq w<1$, $\{n_{i,\,t}\}_{t\ge 0}$ is a nonnegative integer sequence and satisfies 
$n_{i,t}\leq n_{i,t+1}\leq n_{i,t}+N $ with integer $N>0$, then it holds that 
for $0\leq M_1\leq M_2\leq M$, 
\begin{align}\label{4.5}
 \sum_{i=1}^{m}\sum_{j=n_{i,M_1}+1}^{n_{i,M_2+1}} A_{i,j}\leq c_7
 \sum_{t=M_1}^{M_2}B_t+ c_8
 (\sum_{i=1}^{m}
 \sum_{n_i,M_1-1}^{n_{i,M_1}} A_{i,j}+\sum_{i=1}^{m} (C_{i,n_{i,M_1-1}}-C_{i,n_{i,M_2}})),
\end{align}
where $c_7=\frac{4mN}{\underline \alpha(1-w)^2}$, $c_8= \sqrt{m}+\frac{4 w\sqrt{\bar\alpha mN}}{(1-w)\sqrt{\underline\alpha}}$.
Moreover, when $\sum_{t=1}^{\infty}B_t<\infty$, 
$\lim_{t\to\infty}n_{i,t}=\infty$, then it has that
\[ \sum_{j=1}^{\infty} A_{i,j}<\infty,\,\forall i.
\]
\end{lemma}
\begin{proof}
 The proof is given in Appendix \ref{sec7}. 
\end{proof} 
Now we give the global convergence  and the local convergence rate of the BPIREe-$\ell_p$ algorithm.

\begin{theorem}\label{theorem4.2}
Suppose $\{x^k\}$ is the sequence generated by BPIREe-$\ell_p$ algorithm,  	 
and for each block $s_i$, $\nabla_{{s_i}}f(x)$ is Lipschitz continuous within $B_{4\rho}(x^*)$ with respect to $x$, i.e., the inequality (\ref{2.02}) holds.
Under the Assumptions 
\ref{assumption0}-\ref{assumption_4}, it holds that
\[\lim_{k\to+\infty} x^k =x^*. \]
\end{theorem}
\begin{proof} 
 Apply Lemma \ref{lemma4.1}	with $A_{i,j}=\Vert\tilde x^{j-1}_i- \tilde x^{j}_i \Vert$, $\alpha_{i,j}={\tilde L^{j}_i/4} 
 	$, $C_{i,j}=\varepsilon^j_{i}$, $w=\delta$ and
 	$B_t=\phi(F(x^{3tT})-F(x^*))-\phi(F(x^{3(t+1)T})-F(x^*))$, we can 
		obtain the result based on  Theorem 2 in \cite{Wang H}.
\end{proof}
\begin{theorem}\label{theorem 4.3}
	Under the condition of Theorem \ref{theorem4.2}, the concave function is chosen as  $\phi(s)=cs^{1-\theta}(c>0)$, then it holds that
\begin{enumerate}[\rm(i)]
	\item if $\theta = 0$, then
	 there exists $k_0\in\mathbb{N}$ satisfying $x^k\equiv x^*$ for any $k\ge k_0;$  
	\item if $\theta\in(0,\frac{1}{2}]$, $\Vert x^k-x^* \Vert \leq c_{10}\alpha^k-c_{11}\Vert\varepsilon^{d^{{{3(k-1)T}}}}\Vert_1	
	$, $\forall k\in\mathbb{N}$, for a certain $c_{10},c_{11}>0$,  $\alpha\in[0,1);$	
	\item if $\theta \in (\frac{1}{2},1)$, $\Vert x^k -x^*\Vert \leq c_{12}k^\frac{-(1-\theta)}{2\theta-1}-c_{13}\Vert\varepsilon^{d^{{{3(t-1)T}}}}\Vert_1$, $\forall k\in\mathbb{N}$, for a certain $c_{11},c_{12}>0$.
\end{enumerate}
\end{theorem}
\begin{proof}
  	 When $\theta=0$, 
  	 the conclusion can be obtained similarly
 to  Theorem 4 in \cite{XuY},so we leave it out here.
 Next, we focus on the case of $\theta\in(0,1)$. If $F(x^{k_0})=F(x^*)$ for some $k_0$, we deduce the result as the case of (i). Below we suppose $F(x^{k})>F(x^*)$.  Take 
 \[ A_t=\sum_{i=1}^{m}\sum_{j=d^{{{3tT}}}_{s_i}+1}^{\infty}
 \Vert\tilde x^{j-1}_{s_i}-\tilde x^{j}_{s_i} \Vert, 
 \]
 then
 \[ A_{t-1}-A_{t}=\sum_{i=1}^{m}\sum_{i=d^{{{3(t-1)T}}}_{s_i}+1}^{d^{{{3tT}}}_{s_i}}
 \Vert\tilde x^{j-1}_{s_i}-\tilde x^{j}_{s_i} \Vert. 
 \]
 According to (\ref{4.03}) and KL inequality (\ref{A1}), let $k=3tT$, then it holds that
 \begin{align}\label{4.9}
 	c(1-\theta)(F(x^{3tT})-F(x^*))^{-\theta}\ge
 	c_6(A_{t-1}-A_{t}
 	+  \Vert\varepsilon^{d^{{{3(t-1)T}}}}\Vert_1-
 	\Vert\varepsilon^{d^{{{3tT}}}}\Vert_1
 	).
 \end{align}	
 Let $\phi_t=\phi(F(x^{3tT})-F(x^*))$, then the 	
 inequality (\ref{4.9}) implies that
 \begin{align}\label{4.10} 
 	\phi_t=c\phi(F(x^{3tT})-F(x^*))^{1-\theta}\leq    
 	c_{14}(A_{t-1}-A_{t}
 	+  \Vert\varepsilon^{d^{{{3(t-1)T}}}}\Vert_1-
 	\Vert\varepsilon^{d^{{{3tT}}}}\Vert_1 
 	)^{\frac{1-\theta}{\theta}},
 \end{align}
 where $c_{14}=c(c(1-\theta))^{\frac{1-\theta}{\theta}}$.
 Setting $N=t$ in (\ref{4.5}), for any integer $M\ge t$, we infer that
 \begin{align}\label{4.11}
 	\sum_{i=1}^{m}{\sum_{j=d^{3tT}_{s_i}+1}^{d^{3(M+1)T}_{s_i}}}\Vert \tilde x^{j-1}_{s_i} -\tilde x^{j}_{s_i} \Vert
 	\leq c_7 \phi_t + c_8
 	\sum_{i=1}^{m}{\sum_{j=d^{3(t-1)T}_{s_i}+1}^{d^{3tT}_{s_i}}}(\Vert \tilde x^{j-1}_{s_i} -\tilde x^{j}_{s_i} \Vert + \Vert\varepsilon^{j-1}_{s_i}\Vert_1-\Vert\varepsilon^{j}_{s_i}\Vert_1).
 \end{align}
 Letting $M\to+\infty$, 
 together with (\ref{4.10}) and (\ref{4.11}),  we deduce that 
 \[ A_t\leq c_7c_{14} (A_{t-1}-A_{t}
 +  \Vert\varepsilon^{d^{{{3(t-1)T}}}}\Vert_1-
 \Vert\varepsilon^{d^{{{3tT}}}}\Vert_1 
 )^{\frac{1-\theta}{\theta}} 
 +c_8(A_{t-1}-A_t + 
 \Vert\varepsilon^{d^{{{3(t-1)T}}}}\Vert_1-
 \Vert\varepsilon^{d^{{{3tT}}}}\Vert_1), 
 \]
 then we get 
\begin{align*}
A_t+\frac{\sqrt\mu}{1-\sqrt\mu}\Vert\varepsilon^{d^{{{3(t-1)T}}}}\Vert_1
\leq c_7c_{14} (A_{t-1}-A_{t}
+  \Vert\varepsilon^{d^{{{3(t-1)T}}}}\Vert_1
)^{\frac{1-\theta}{\theta}} +c_8
(A_{t-1}-A_t + 
  \Vert\varepsilon^{d^{{{3(t-1)T}}}}\Vert_1).
\end{align*}
Since $\varepsilon^k_i\leq\sqrt{\mu}\varepsilon^{k-1}_i $, $i\in\{1,2,\dots,n\}$,	 
  it holds that 
  $\varepsilon^k_{i}\leq \frac{\sqrt\mu}{1-\sqrt\mu}
  (\varepsilon^{k-1}_{i}-\varepsilon^k_{i})$.
    Thus, we have
  \begin{align*}
  & A_t+\frac{\sqrt\mu}{1-\sqrt\mu}\Vert\varepsilon^{d^{{{3(t-1)T}}}}\Vert_1  \\
  &	\leq c_7c_{14}(A_{t-1}-A_{t}
  	+ \frac{\sqrt\mu}{1-\sqrt\mu}( \Vert\varepsilon^{d^{{{3(t-1)T}}-1}}\Vert_1-
  	\Vert\varepsilon^{d^{{{3(t-1)T}}}}\Vert_1
  	))^{\frac{1-\theta}{\theta}}  \\
   & \quad +c_8
  	(A_{t-1}-A_t + 
  	\frac{\sqrt\mu}{1-\sqrt\mu}( \Vert\varepsilon^{d^{{{3(t-1)T}}-1}}\Vert_1-
  	\Vert\varepsilon^{d^{{{3(t-1)T}}}}\Vert_1).
  \end{align*}
Letting $t\to+\infty$, we get
  $A_{t-1}+\frac{\sqrt\mu}{1-\sqrt\mu}\Vert\varepsilon^{d^{{{3(t-1)T}-1}}}\Vert_1-
(A_t+\frac{\sqrt\mu}{1-\sqrt\mu}\Vert\varepsilon^{d^{{{3(t-1)T}}}}\Vert_1)\leq 1
 $. Then we obtain the desired result based on the Lemma 3 in \cite{XuY}.   
\end{proof}

  \section{Numerical Experiments }\label{sec4}
   To verify the efficiency of the proposed algorithm, we consider two examples that are generated  on random data sets 
  in the following.
  The initial points in the experiments are set as the original. 
  
 \begin{example}
 	We consider the following log regularized least squares problem, which is also considered in \cite{Yu P}, that is 
 	\begin{align}\label{001}
 		\min_{x\in\mathbb{R}^q} F(x):=\frac{1}{2}\Vert Ax-b\Vert^2 
 		+\sum_{i=1}^{q}(\lambda {\rm{log}}(\vert x_i\vert+\epsilon )-\lambda {\rm{log}}\epsilon), 
  	\end{align}
 	where 
 	$A\in\mathbb{R}^{n\times q}$, $b\in\mathbb{R}^n$, $\lambda>0$, $\epsilon>0$. 
 \end{example}
 In this case, $s_j=\{j\}$, $f(x)=\frac{1}{2}\Vert Ax-b\Vert^2$,  $h(t)={\rm{log}}(t+\epsilon)$, $g(y)=\vert y\vert$, where $x\in\mathbb{R}^q$, $t,y\in\mathbb{R}$.  In this case,
 The BPIREe algorithm is denoted as PIREe, and we compare it  
 with ${\rm{IRL_1}}$ and  ${\rm{IRL_1e_1}}$ \cite{Yu P}.  
 The stopping criterion is set as
 \begin{align}\label{00}
 	\Vert x^{k+1}-x^k \Vert/\Vert x^k \Vert<10^{-4},
 \end{align} 
 and the relative error is used as a metric to illustrate the efficiency of the proposed algorithm which is defined as follows
 \[ {\rm{rel.err.}}:= \Vert x^{k}-x^* \Vert/\Vert x^k\Vert. 
 \]
 
   The parameters of our algorithm are set as:     
 $n=1000,\, q=3000$, ${L_f=\rm{norm(A*A^T)}}$, $\epsilon=0.1$ and the tradeoff parameter
 $\lambda$ is set as $\lambda=5e-4$.  
 We test the algorithm on two different data types.
  In detail,
 the first one is that the matrix $A\in\mathbb{R}^{n\times q}$ is generated with i.i.d standard Gaussian entries and then normalizes this matrix to get the unit column norm.   
 The observation vector $b\in\mathbb{R}^q$ is generated by $b=Ax^*+0.001e$ where $x^*\in\mathbb{R}^q$ is a sparse vector with $\Vert x^*\Vert_0=50$, and $e$ is a Gaussian random vector.  For the second one, the sensing matrix $A\in\mathbb{R}^{p\times q}$  is ill-conditioned and is generated with $A=U\Sigma V^T$   where $\Sigma$ is the diagonal matrix with the $i$th diagonal element
 is $\sigma_i=10^{-4}+(i-1)/10$ for $i=1,2,\dots,n$,
 $U\in\mathbb{R}^{n\times n}$ and $V\in\mathbb{R}^{q\times n}$  are orthogonal matrices with orthogonal columns. The vector $b$ is generated as the same as in the former case.

     The extrapolation parameter in ${\rm{IRL_1e_1}}$ is chosen as the same as the FISTA with fixed restart \cite{ODonoghue}:   
   \begin{align}\label{A}
   	\beta_{k-1}=\frac{t_{k-1}-1}{t_{k}}\quad {\rm{with}}\quad t_{k }=\frac{1+\sqrt{1+4{t^2_{k-1}}}}{2}, 
   \end{align}
   and we reset $t_0=t_{-1}=1$ every  $N=200$ iterations. For the PIREe method, the extrapolation parameter is chosen as in FISTA. But when the objective function value is increasing,    
   we set the 
   extrapolation parameter is zero and do the $k$th iteration again.   
  
   The numerical results are given in Figures \ref{fig1}, \ref{fig2}. In part (a) of each figure, it shows the number of iterations against  
  $\vert F(x^k)-F(\bar x)\vert$ where $\bar{x}$ is the output of the proposed algorithm.   Besides, in part (b) of each figure, 
  we plot 
  $\Vert x^k-\bar{x} \Vert/\Vert\bar{x}\Vert$ against the number of the iterations.    
  From the displayed results in  Figures \ref{fig1}, \ref{fig2},  no matter when $A$ is normal or ill-conditioned,   
  we observe that the numerical results of PIREe are better than  ${\rm{IRL_1e_1}}$ and ${\rm{IRL_1}}$.    
  Moreover,  compared with ${\rm{IRL_1e_1}}$ and  ${\rm{IRL_1}}$, the used time and the relative error of PIREe are slightly better than ${\rm{IRL_1e_1}}$ and better than ${\rm{IRL_1}}$ from Table \ref{table0}.

 \begin{figure}[h]
 	\footnotesize
 	\centering
 	\begin{minipage}[t] {0.5 \textwidth}
 		\includegraphics[scale=0.06]{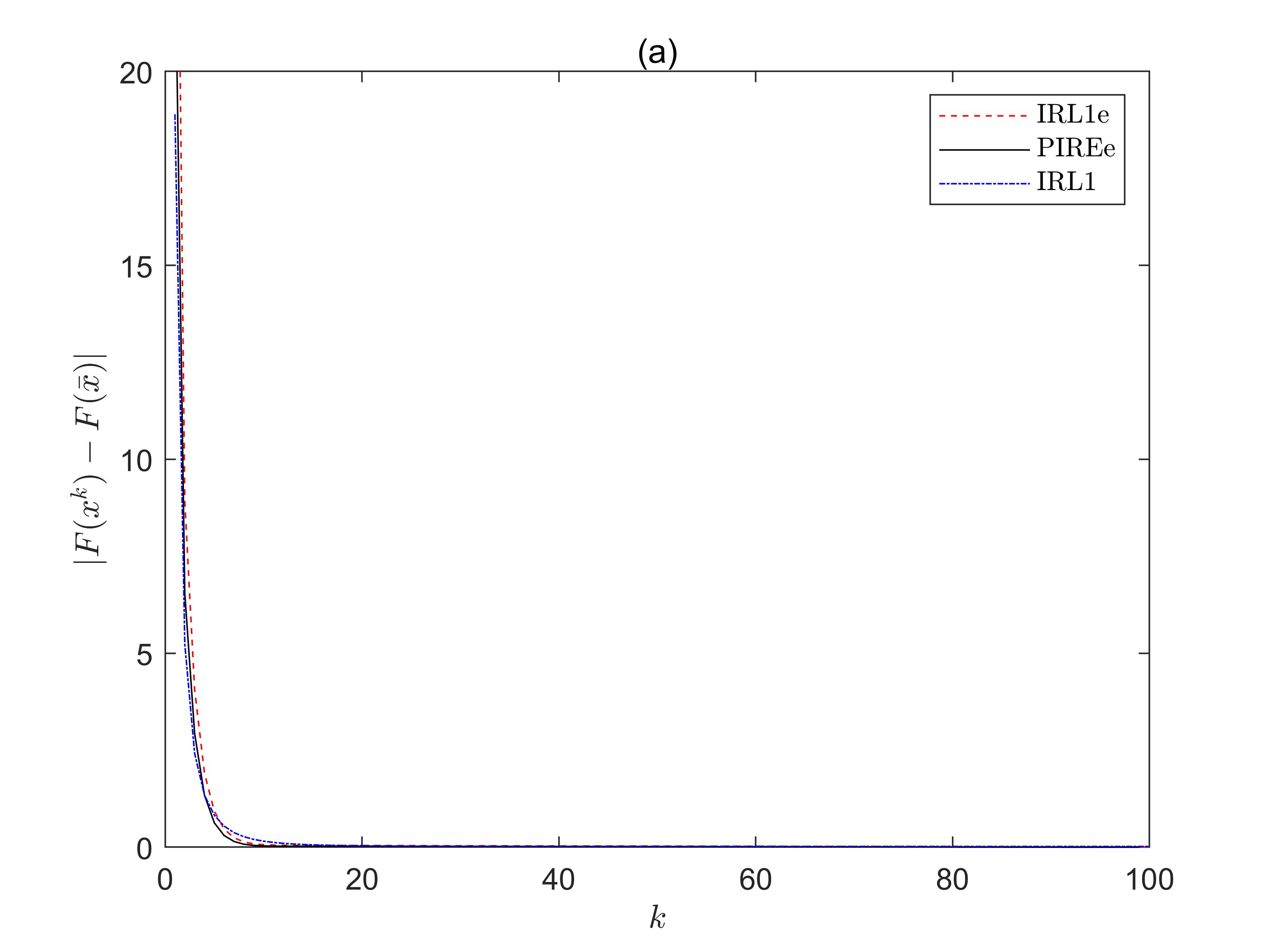}
 		\end {minipage}
 		\centering
 		\begin{minipage}[t] {0.4\textwidth}
 			\includegraphics[scale=0.06]{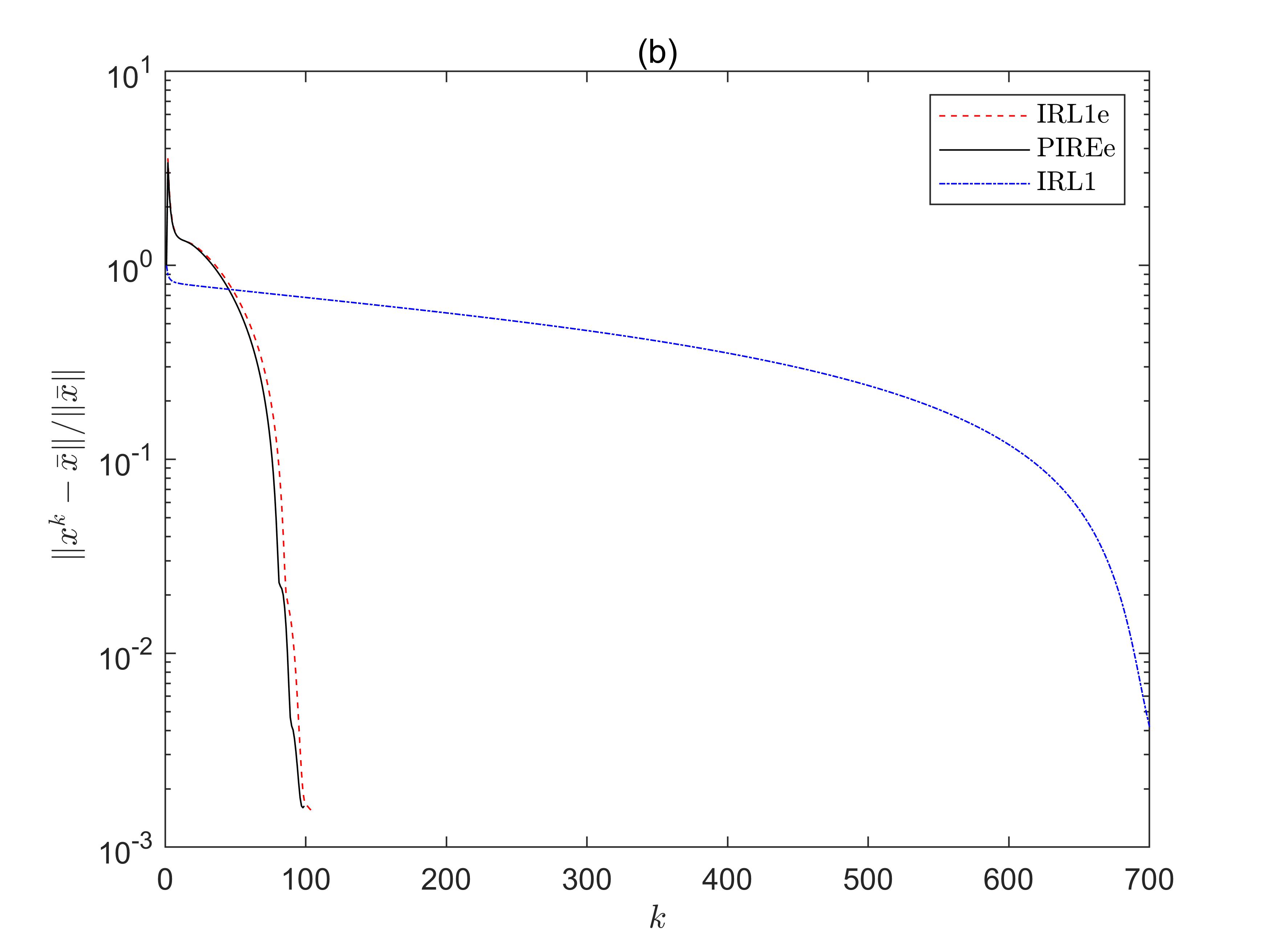}
 		\end{minipage}
 		\centering\caption{Numerical results when the sensing matrix  $A$ is well-conditioned}\label{fig1}
 	\end{figure} 
  \begin{figure}[h]
 	\footnotesize
 	\centering
 	\begin{minipage}[t] {0.5\textwidth}
 		\includegraphics[scale=0.06]{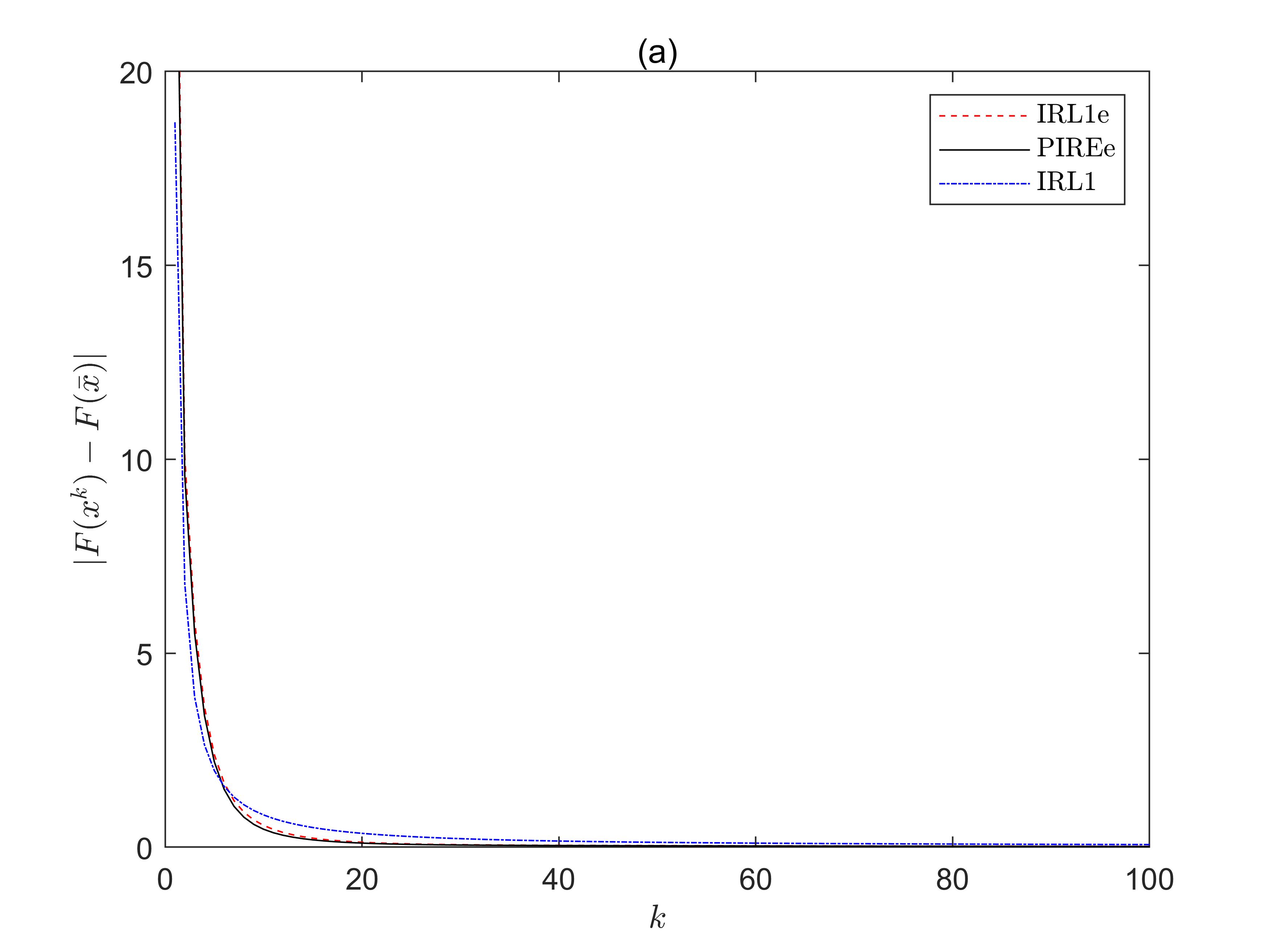}
 		\end {minipage}
 		\centering
 		\begin{minipage}[t] {0.4\textwidth}
 			\includegraphics[scale=0.06]{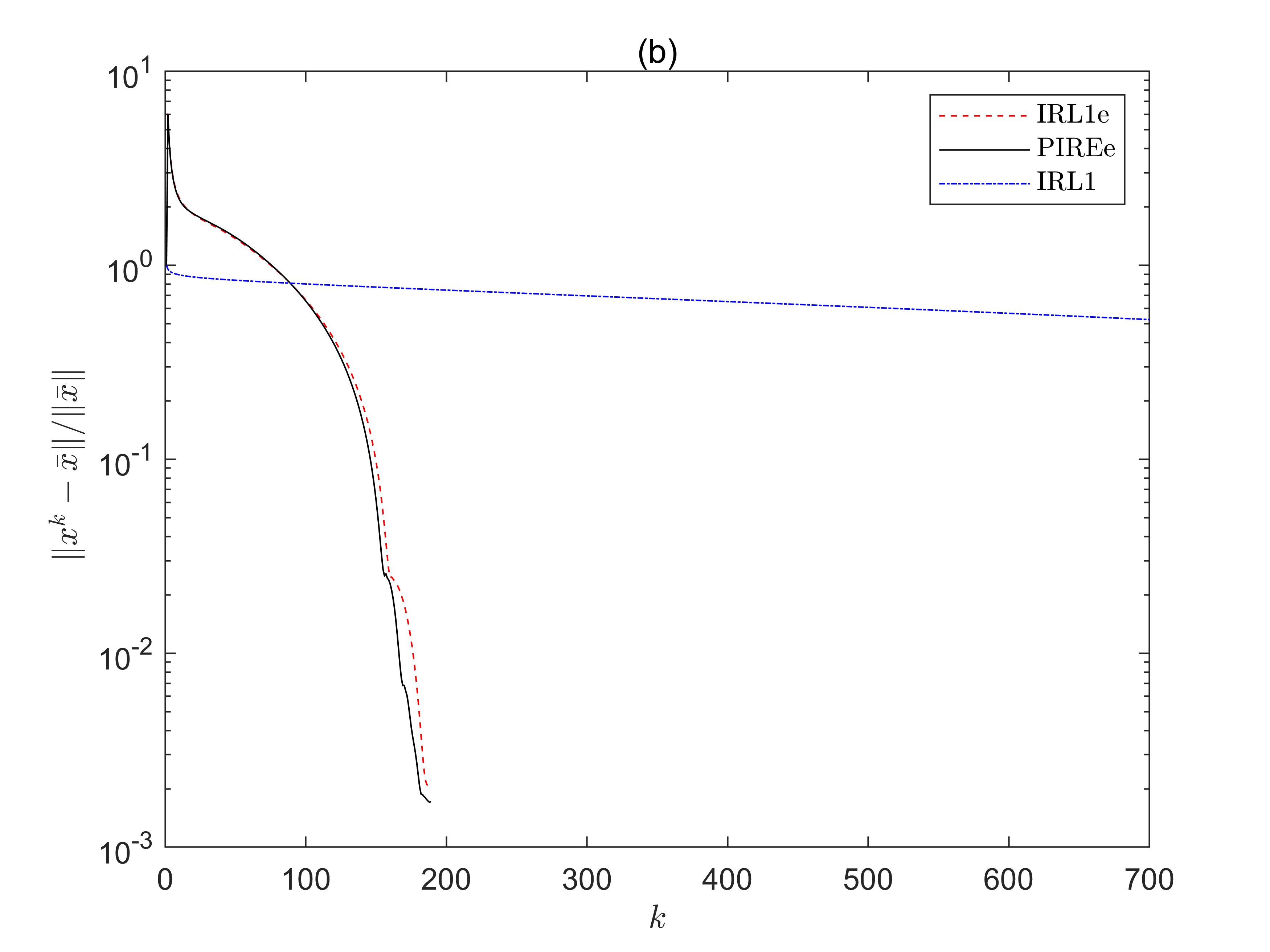}
 		\end{minipage}
 		\centering\caption{Numerical results when the sensing matrix $A$ is ill-conditioned}\label{fig2}
 	\end{figure}
 
  \begin{table}[htb]
  	
 	\begin{center}
 	\caption{Numerical results of PIREe,   ${\rm{IRL_1e_1}}$ and   ${\rm{IRL_1}}$}\label{table0}
 	\begin{tabular}{ccccccc} 
 		\toprule
 	  \multirow{2}{*}{$A$} & \multicolumn{3}{l}{time}& \multicolumn{3}{l}{rel-err}     \\
 		\cmidrule(lr){2-4}\cmidrule(lr){5-7}
 		    &  PIREe&  ${\rm{IRL_1e_1}}$ &${\rm{IRL_1}}$& PIREe & ${\rm{IRL_1e_1}}$ & ${\rm{IRL_1}}$ \\
 		\midrule
 	    well-conditioned  &0.86  &0.92   &3.83   &1.71e-02 &1.94e-02 & 2.13e-02  \\
 	ill-conditioned  & 1.42    &1.43   &18.72   &1.84e-02 &    2.14e-02   &    4.09e-02    \\	 
 		\bottomrule
 	\end{tabular}
 \end{center}
 \end{table}

  \begin{example}
 	We consider the following approximate $\ell_q$ optimization problem:
 	\begin{align}\label{002}	
 		\min_{X\in\mathbb{R}^{q\times t}} F(X): =
 		\frac{1}{2}\Vert AX-B \Vert^2_F +\sum_{i,j}
 		(\vert X_{i,j}\vert+\varepsilon)^p,\, 0<p<1,
 	\end{align}
 	where $A\in\mathbb{R}^{n\times q}$,  $B\in\mathbb{R}^{n\times t}$, $X=(X_{i,j})\in\mathbb{R}^{q\times t}, i=1,\dots,q,\,j=1,\dots,t$.  
 \end{example} 
 In this example, $f(X)=\frac{1}{2}\Vert AX-B \Vert^2_F$,  $h(y)=\lambda(y+\varepsilon)^p$, $g(t)=\vert t\vert$.  The matrices $A, B$ are generated just like the well-conditioned case in the first example. We test the algorithm on two different sizes of data sets, i.e.,  $n=100,q=500,t=50$; $n=300,q=1000,t=200$. The corresponding blocks are set as  $m=10$ and $m=20$, respectively. In each matrix column $X^*$, the sparsity is set as $0.02q$.     The parameter $\lambda,\mu,q$ is chosen as  $\lambda=0.015,\,\mu=0.1,\, p=0.1$. 
 The extrapolation parameter is chosen as (\ref{A}).   The extrapolation parameter is  
    set as zero if the objective function value is increasing at $k$th iteration 
  and we do the iteration again.     
 
 We compare the proposed algorithm BPIREe with 
 PIRE-AU and PIRE-PS \cite{Sun T}. The relative error is used as an evaluation metric  which is defined as 
 \[ {\rm{rel.err.}}:= \frac{\Vert X^{k}-\bar X \Vert_F}{\Vert X^k \Vert_F}.  
 \]  
 The stopping criterion is set as 
 \begin{align}\label{01}
 	\frac{\Vert X^{k+1}-X^k \Vert_F}{\Vert X^k \Vert_F}<10^{-4}.
 \end{align}
 
   The computational results are presented in Figures \ref{fig3},\ref{fig4}. The curves of  $\vert F(X^k)-F(\bar X)\vert$ and 
 $\Vert X^k-\bar{X} \Vert_F/\Vert\bar{X}\Vert_F$ against the number of the iterations are described respectively  where $\bar{X}$ is the output of the proposed algorithm.  
 From them, we see that the BPIREe algorithm outperforms  PIRE-AU and PIRE-PS.
 In addition,  we present the comparison results 
 concerning the relative error and the used time of BPIREe, PIRE-AU, and PIRE-PS in
 Table \ref{table1}. It is shown that the used time of the BPIREe algorithm is the least. 
    \begin{figure}[h]
  	\footnotesize
  	\centering
  	\begin{minipage}[t] {0.5 \textwidth}
  		\includegraphics[scale=0.06]{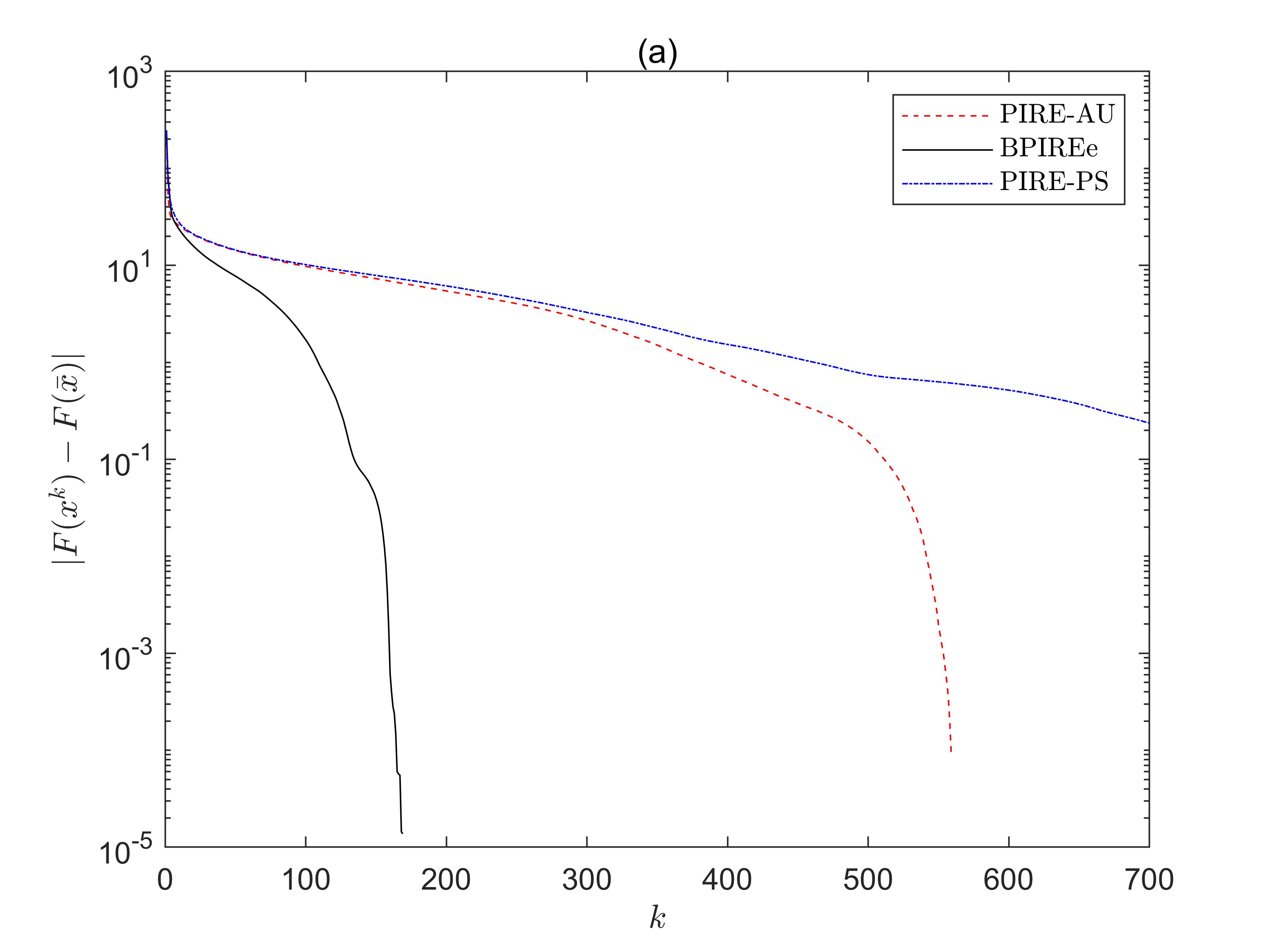}
  		\end {minipage}
  		\centering
  		\begin{minipage}[t] {0.4\textwidth}
  			\includegraphics[scale=0.06]{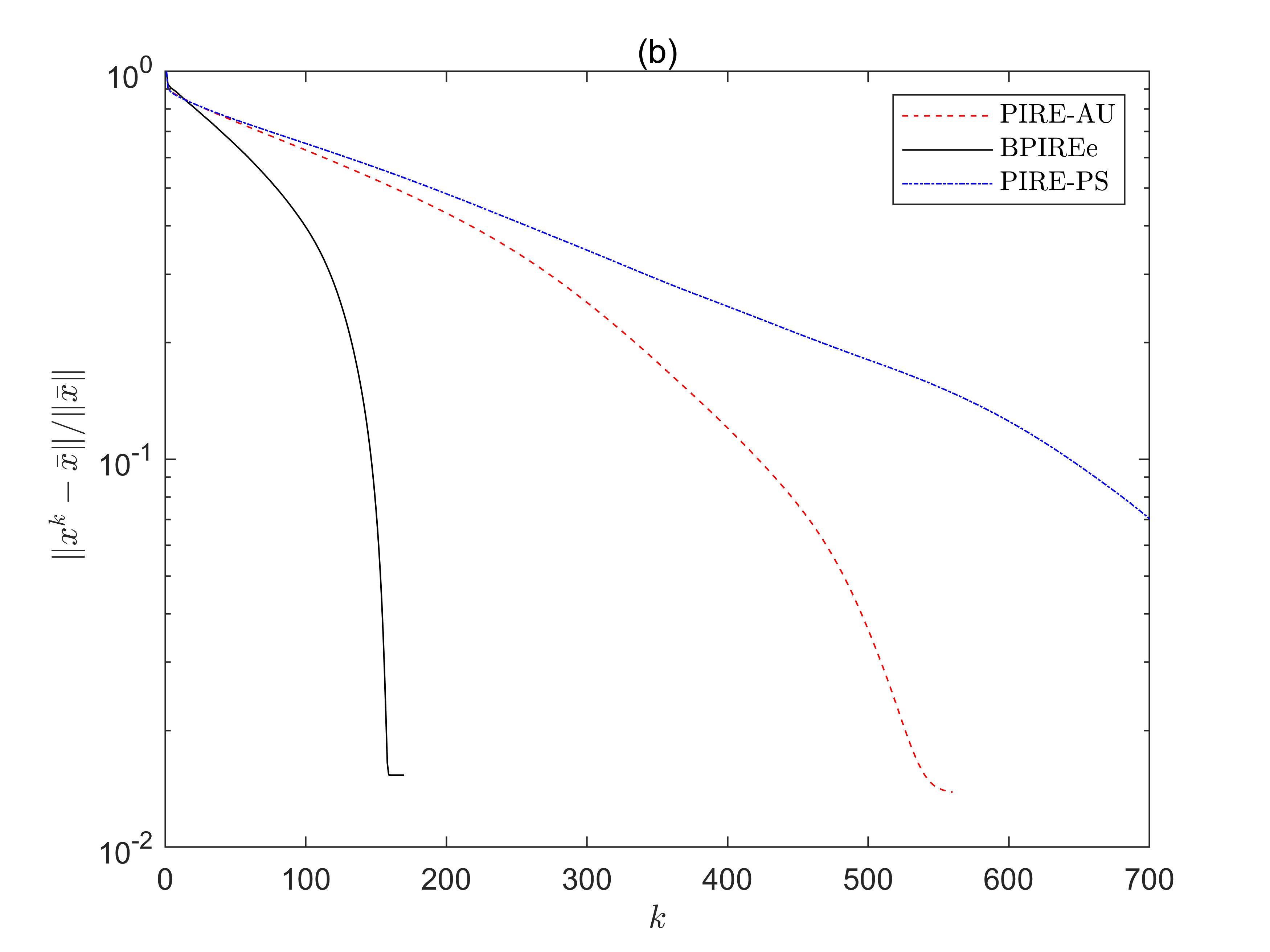}
  		\end{minipage}
  		\centering\caption{Numerical results when $n=100,q=500,t=50$}\label{fig3}
  	\end{figure}
   \begin{figure}[h]
  	\footnotesize
  	\centering
  	\begin{minipage}[t] {0.5\textwidth}
  		\includegraphics[scale=0.06]{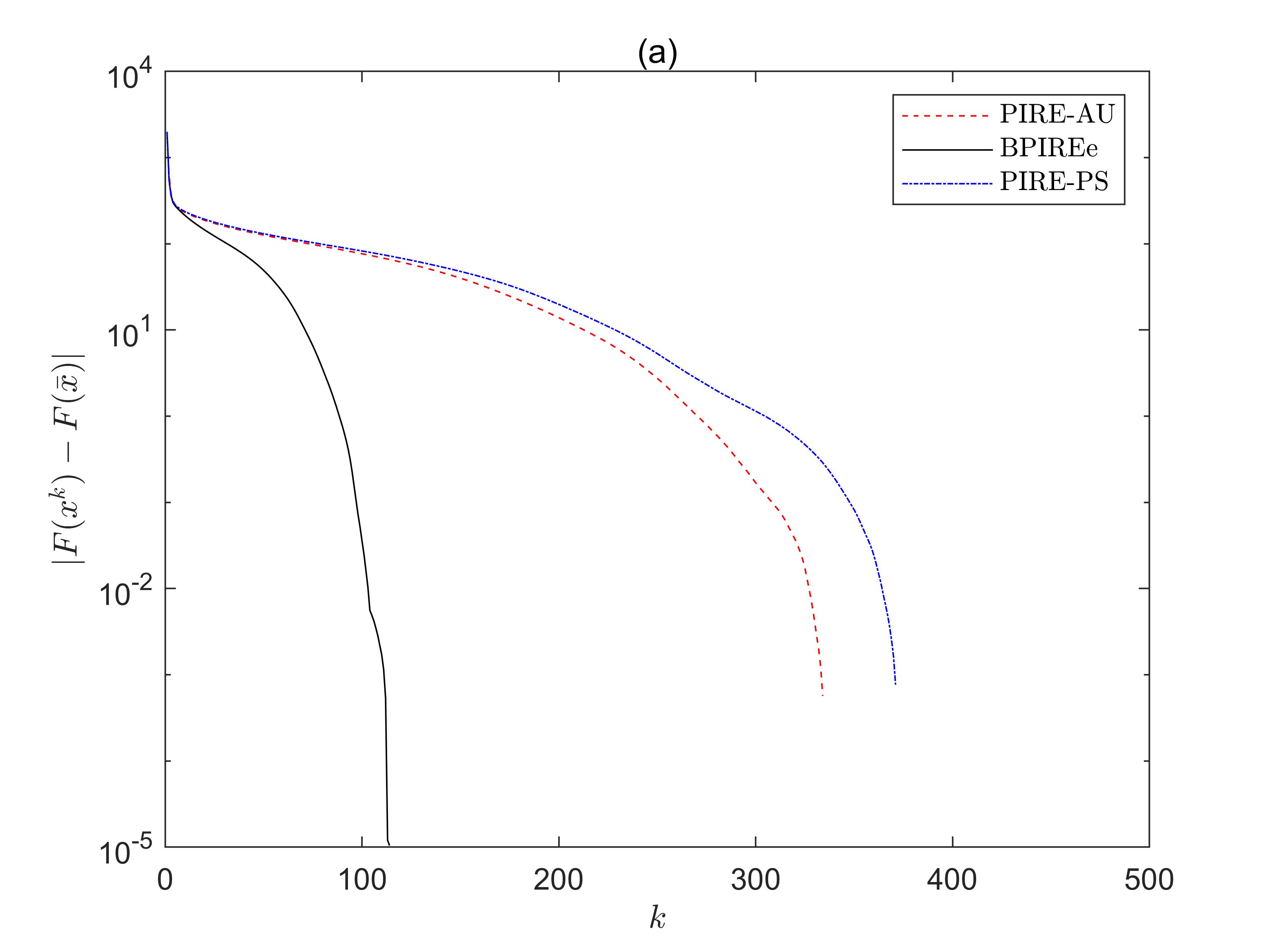}
  		\end {minipage}
  		\centering
  		\begin{minipage}[t] {0.4\textwidth}
  			\includegraphics[scale=0.06]{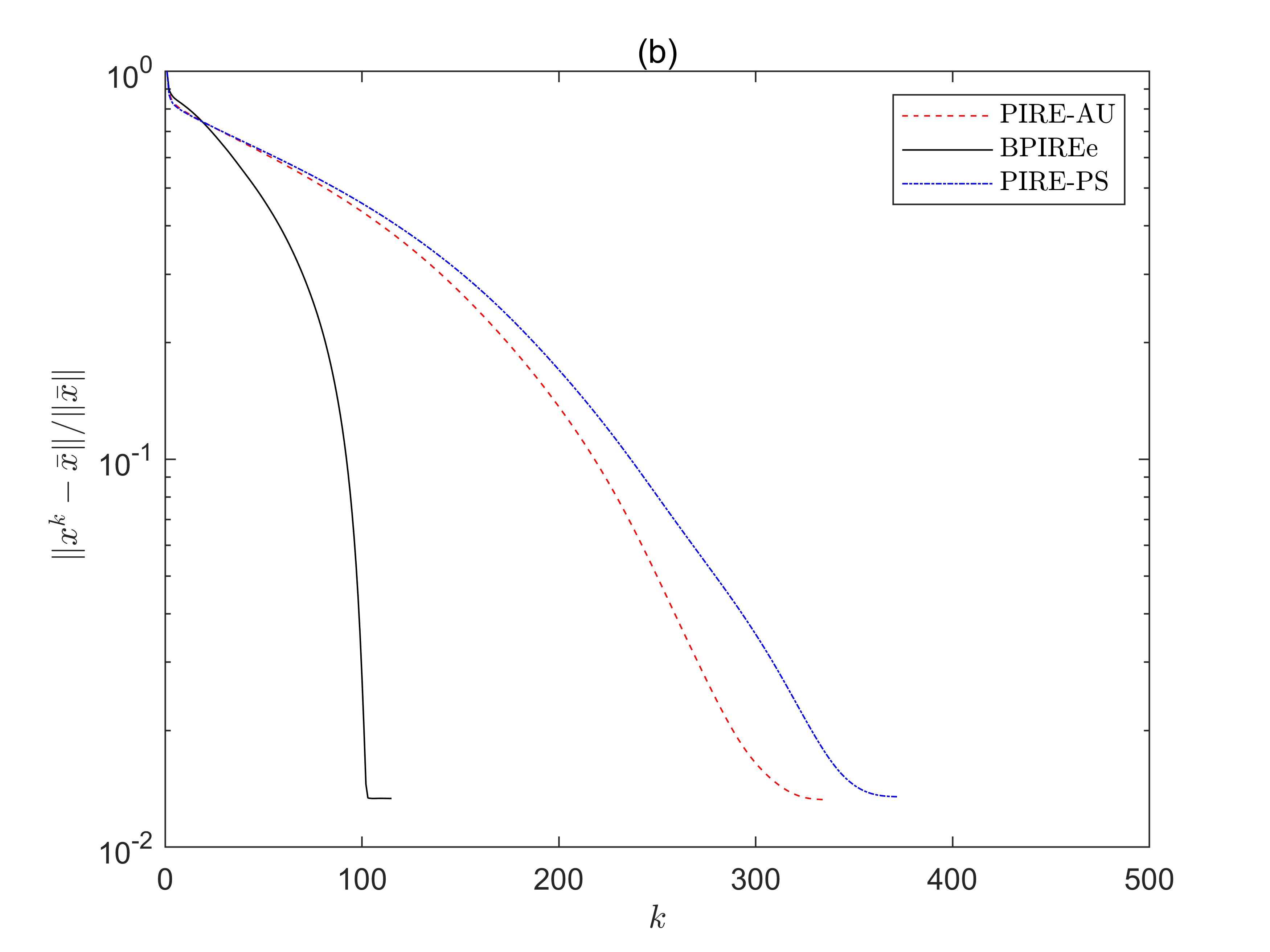}
  		\end{minipage}
  		\centering\caption{Numerical results when   $n=300,q=1000,t=200$ }\label{fig4}
  	\end{figure}

 \begin{table}[htb]
		\centering\caption{Numerical results of BPIREe, PIRE-AU and PIRE-PS}\label{table1}
	\begin{tabular}{ccccccccc} 
		\toprule
		\multirow{2}{*}{$m$} &\multirow{2}{*}{$n$} & \multirow{2}{*}{$t$} & \multicolumn{3}{l}{time}& \multicolumn{3}{l}{rel-err}     \\
		\cmidrule(lr){4-6}\cmidrule(lr){7-9}
		&&  & BPIREe& PIRE-AU & PIRE-PS& BPIREe & PIRE-AU& PIRE-PS \\
		\midrule
		100&500&50 &1.11  &3.22   &3.43   &1.29e-02 &1.44e-02 & 1.28e-02  \\
		300&1000&200 & 15.21    &35.60   &36.42   &1.32e-02 &    1.33e-02   &    1.32e-02    \\	 
		\bottomrule
	\end{tabular}
	\centering
\end{table}

\section{Conclusions}\label{sec5}
    We have developed a block proximal iteratively reweighted algorithm with extrapolation for solving a class of nonsmooth nonconvex problems. The proposed algorithm can be used to solve the $\ell_p$ regularization problem by employing a special adaptively updating strategy. 
    There exists the non-zero extrapolation parameter that ensures  the objective function is nonincreasing. 
    We obtain the global convergence of the algorithm when the objective function satisfies the KL property. 
    The numerical results illustrate that
    the extrapolation item speeds up the algorithm's convergence, and 
    the iterations and the used time are less than others.

\section*{Acknowledgement}
	This work is supported by National Natural Science Foundation of China(No.11991024, 11971084) and it is also supported by Natural Science Foundation of Chongqing(No.cstc2018jcyj-yszxX0009, cstc2019jcyj-zdxmX0016).

\section*{Declarations}	
\textbf{Confict of interest} The authors declare that they have no con ict of interest.
\bibliographystyle{plain}

\begin{thebibliography}{99}
	
\bibitem{Hedy Attouch}
H. Attouch, J. Bolte, and B. F.  Svaiter, \emph{Convergence of descent methods for semi-algebraic and tame problems: proximal algorithms, forward-backward splitting, and regularized Gauss-Seidel methods}, Math. Program. 137(1)(2013), pp. 91-129.

	
	
\bibitem{Attouch}
	H. Attouch, J. Bolte, P. Redont, and A. Soubeyran, \emph{ Proximal alternating minimization and projection methods for nonconvex problems: an approach based on the Kurdyka-Lojasiewicz
	inequality,} Math. Oper. Res. 35(2)(2010), pp. 438-457. 
	
   \bibitem{Auslender}
A. Auslender and M. Teboulle, \emph{Interior gradient and proximal methods for convex and conic optimization,} SIAM J. Optim. 16(2006), pp. 697-725.  
	
	
	

 \bibitem{Beck}
A. Beck and M. Teboulle, \emph{ A fast iterative shrinkage-thresholding algorithm for linear inverse problems,} SIAM J. Imaging Sci,   2(1)(2009), pp. 183-202.	
	
\bibitem{Bolte J}
J. Bolte, A. Daniilidis, and A. Lewis, \emph{The {\L}ojasiewicz inequality for nonsmooth subanalytic functions with
applications to subgradient dynamical systems,} SIAM J. Optim. 17(4)(2007), pp. 1205-1223.	
	
\bibitem{Bolte}
J. Bolte, S. Sabach, and M. Teboulle,\emph{Proximal alternating linearized minimization for nonconvex and nonsmooth problems,} Math. Program. 146 (2014), pp. 459-494.	

\bibitem{Candes}
E. J. Candes, M. B. Wakin, and S. P. Boyd, \emph{ Enhancing sparsity by reweighted $\ell_1$ minimization,} J. Fourier Anal. Appl., 14(5)(2008), pp. 877-905.

\bibitem{Chen X} 
X. Chen, F. Xu, and Y. Ye, \emph{Lower bound theory of nonzero entries in solutions of $\ell_2-\ell_p$ minimization,}
SIAM J. Sci. Comput. 32(2010), pp. 2832-2852.



\bibitem{Daubechies}
I. Daubechies, R. DeVore,  M. Fornasier, and C.S. Gunturk,  \emph{Iteratively reweighted least squares minimization for sparse recovery,} Commun. Pure. Appl. Math. 63(2010), pp. 1-38. 

   \bibitem{Drusvyatskiy}
D. Drusvyatskiy and C. Paquette, \emph{Efficiency of minimizing compositions of convex functions and smooth maps,}  Math. Program.  178(1)(2019), pp. 503-558.

\bibitem{Chartrand}
R. Chartrand and W. Yin, \emph{Iteratively reweighted algorithms for compressive sensing,} In: 33rd International Conference on Acoustics, Speech, and Signal Processing (ICASSP) (2008).

\bibitem{Ghadimi}
S. Ghadimi and G. Lan,\emph{ Accelerated gradient methods for nonconvex nonlinear and stochastic 
programming,} Math. Program. 156(2016), pp. 59-99.

\bibitem{Kurdyka K}
K. Kurdyka, \emph{On gradients of functions definable in o-minimal structures,} Annales de l'institut Fourier.
48(3)(1998), pp. 769-783. 


   
\bibitem{LanG}
G. Lan, Z. Lu, and R.D.C. Monteiro, \emph{Primal-dual first-order methods with $O(1/{\epsilon})$ iteration-complexity  for cone programming,} Math. Program. 126(2011), pp. 1-29.

\bibitem{LuC} 
C. Lu,  Y. Wei,  Z. Lin, and S. Yan, \emph{Proximal iteratively reweighted algorithm with multiple splitting for nonconvex sparsity optimization,} In: Proceedings of the AAAI Conference on Artificial Intelligence
(2014).


\bibitem{LuZ} 
Z. Lu, \emph{Iterative reweighted minimization methods for  $\ell_p $  regularized unconstrained nonlinear programming,}  Math. Program.  147(1)(2014), pp. 277-307. 

\bibitem{Mordukhovich}
 B. S. Mordukhovich,\emph{ Variational analysis and generalized differentiation II: Applications[M],} Springer, Berlin, 2006.
 
 \bibitem{ODonoghue} 
 B. O'Donoghue, E.J. Cand$\grave{e}$s, \emph{Adaptive restart for accelerated gradient schemes,} Found. Comput.
 Math. 15(2015), pp. 715-732.
 
 \bibitem{Rockafellar} 
 R. T. Rockafellar, \emph{Convex Analysis}, Princeton University Press, Princeton, 2015.
 
\bibitem{Rockafellar}
R. T. Rockafellar, and R. J. B. Wets,\emph{ Variational Analysis,} Springer, New York, 2009.



\bibitem{Yurii Nesterov}
Y. Nesterov, \emph{Introductory lectures on convex programming volume i: Basic course,}
Lecture notes, 3(4)(1998). 


\bibitem{Sun T}
T. Sun, H. Jiang, and L. Cheng, \emph{Global convergence of proximal iteratively reweighted algorithm,} J. Global Optim. 68(4)(2017), pp. 815-826.

 
\bibitem{Tseng}
P. Tseng, \emph{Approximation accuracy, gradient methods, and error bound for structured
convex optimization,} Math. Program. 125(2)(2010), pp. 263-295.

 

 


 

 





 
\bibitem{Wang H}
H. Wang, H. Zeng , and J. Wang, \emph{An extrapolated iteratively reweighted $\ell _1 $ method with complexity analysis,} Comput. Optim. Appl.  (2022), pp. 1-31. 
\bibitem{Wang H2} 
H. Wang, H. Zeng, and J. Wang,  \emph{Relating $\ell_p$ regularization and reweighted $\ell_1$ regularization,}  Optim. Lett. (2021), pp. 1-22.



\bibitem{XuY1}
Y. Xu and W. Yin,\emph{ A block coordinate descent method for regularized multiconvex optimization with
applications to nonnegative tensor factorization and completion,} SIAM J. Imaging Sci. 6(3)(2013), pp. 1758-1789.

\bibitem{XuY}
Y. Xu and W. Yin, \emph{A Globally Convergent Algorithm for Nonconvex Optimization Based on Block Coordinate Update,} J. Sci. Comput. 72(2)(2017), pp. 700-734.

\bibitem{Yu P}
P. Yu and T. K. Pong, \emph{Iteratively reweighted $\ell_1$ algorithms with extrapolation,} Comput. Optim. Appl. 73(2)(2019), pp. 353-386.


	


 



 
 
     

 

  

\end{thebibliography}

\begin{appendices}
   \section{Appendix: Proof of Lemma \ref{lemma4.1}}\label{sec7}
   Suppose $a_t$ and $u_t$ be the vectors and their 
   $i$th entries are
    \[ (a_t)_i= \sqrt{\alpha_{i,\,n_{i,t}}} ,\, (u_t)_i=A_{i,\,n_{i,t}}. 
    \]
    The inequality (\ref{4.4}) can be rewritten as 
    \begin{align}\label{a.3}
    &\Vert a_{t+1}\odot u_{t+1} \Vert^2 + (1-w^2)\sum_{i=1}^{m}
    \sum_{j=n_{i,t}+1}^{n_{i,t+1}-1}\alpha_{i,j}A^2_{i,j}  \notag \\
   \leq& w^2\Vert a_{t}\odot u_{t} \Vert^2
     + B_t\sum_{i=1}^m\sum_{j=n_{i,t-1}+1}^{n_{i,t}}
     (A_{i,j}+C_{i,j-1}-C_{i,j}).
    \end{align}
   From the proof of Lemma 2 in \cite{XuY}, we get
  \begin{align}\label{a.2}
   \frac{1+w}{2}\Vert a_{t+1}\odot u_{t+1} \Vert+
  V_1\sum_{i=1}^{m}\sum_{j=n_{i,t}+1}^{n_{i,t+1}-1} A_{i,j}
  \leq\sqrt{\Vert a_{t+1}\odot u_{t+1} \Vert^2+
  \underline\alpha(1-w^2) \sum_{i=1}^{m}\sum_{j=n_{i,t}+1}^{n_{i,t+1}-1} A^2_{i,j}}, 
\end{align}
 where $V_1\leq\sqrt{\frac{\underline\alpha(1-w^2)(4-(1+w)^2)}{4mN}}$.    
Given $V_2>0$, we derive that
\begin{align}\label{a.1}
&\sqrt{w^2\Vert a_{t}\odot u_{t} \Vert^2 + B_t\sum_{i=1}^m\sum_{j=n_{i,t-1}+1}^{n_{i,t}}	(A_{i,j}+C_{i,j-1}-C_{i,j})} \notag \\
 \leq & w \Vert a_{t}\odot u_{t} \Vert
+\sqrt{B_t\sum_{i=1}^{m} \sum_{j=n_{i,t-1}+1}^{n_{i,t}} (A_{i,j}+C_{i,j-1}-C_{i,j})} \notag\\
 \leq & w \Vert a_{t}\odot u_{t} \Vert
 +\sqrt{B_t\sum_{i=1}^{m} \sum_{j=n_{i,t-1}+1}^{n_{i,t}}  A_{i,j}}  
 +\sqrt{B_t \sum_{i=1}^{m} \sum_{j=n_{i,t-1}+1}^{n_{i,t}} (C_{i,j-1}-C_{i,j})}   \notag\\
  \leq &  w\Vert a_{t}\odot u_{t} \Vert +
  2V_2B_t+\frac{1}{4V_2}\sum_{i=1}^{m}\sum_{j=n_{i,t-1}+1}^{n_{i,t}}  A_{i,j}+
  \frac{1}{4V_2}\sum_{i=1}^{m}\sum_{j=n_{i,t-1}+1}^{n_{i,t}} (C_{i,j-1}-C_{i,j})\notag \\
  \leq & w \Vert a_{t}\odot u_{t} \Vert +
  2V_2B_t+\frac{1}{4V_2}\sum_{i=1}^{m}\sum_{j=n_{i,t-1}+1}^{n_{i,t}-1}  A_{i,j}+
  \frac{1}{4V_2}\sum_{i=1}^{m}\sum_{j=n_{i,t-1}+1}^{n_{i,t}} (C_{i,j-1}-C_{i,j})
  +\frac{\sqrt m}{4V_2}\Vert u_t\Vert.  
\end{align}
Together with (\ref{a.3}), (\ref{a.2}) and (\ref{a.1}), we have
\begin{align*}
 &\frac{1+w}{2}\Vert a_{t+1}\odot u_{t+1} \Vert+
V_1\sum_{i=1}^{m}\sum_{j=n_{i,t}+1}^{n_{i,t+1}-1} A_{i,j} \notag\\
\quad &\leq w \Vert a_{t}\odot u_{t} \Vert +
2V_2B_t+\frac{1}{4V_2}\sum_{i=1}^{m}\sum_{j=n_{i,t-1}+1}^{n_{i,t}-1}  A_{i,j}+
\frac{1}{4V_2}\sum_{i=1}^{m}\sum_{j=n_{i,t-1}+1}^{n_{i,t}} (C_{i,j-1}-C_{i,j})
+\frac{\sqrt m}{4V_2}\Vert u_t\Vert
\end{align*}
The rest of the proof is analogous to the Lemma 2 in \cite{Wang H},  so we get the desired result.
  
\end{appendices}
\end{document}